\documentclass{amsart}
\usepackage{graphicx}
\usepackage{amssymb,amscd,amsthm,amsxtra}
\usepackage{latexsym}
\usepackage{epsfig}
\usepackage{mathtools}
\usepackage{esint}
\usepackage{color}

\newtheorem{thm}{Theorem}[section]
\newtheorem{cor}[thm]{Corollary}
\newtheorem{lem}[thm]{Lemma}
\newtheorem{prop}[thm]{Proposition}
\theoremstyle{definition}
\newtheorem{defn}[thm]{Definition}
\theoremstyle{remark}
\newtheorem{rem}[thm]{Remark}
\numberwithin{equation}{section}

\newcommand{\R}{\mathbb R}

\newcommand{\eps}{\epsilon}

\newcommand{\p}{\partial}

\newcommand{\comment}[1]{}

\begin{document}

\title{An energy model for harmonic functions with junctions}
\author{D. De Silva}
\address{Department of Mathematics, Barnard College, Columbia University, New York, NY 10027}
\email{\tt  desilva@math.columbia.edu}
\author{O. Savin}
\address{Department of Mathematics, Columbia University, New York, NY 10027}\email{\tt  savin@math.columbia.edu}
\begin{abstract}We consider an energy model for harmonic graphs with junctions and study the regularity properties of minimizers and their free boundaries. \end{abstract}

\maketitle
\section{Introduction}

In this paper we propose an energy model for $N$ harmonic graphs with junctions and study the regularity properties of the minimizers and their free boundaries. 

The physical motivation is the following. Let $\Omega$ be a bounded domain in $\R^n$ and consider $N \ge 2$ elastic membranes which are represented by the graphs of $N$ real-valued functions $$\{(x,u_i(x))| \quad x \in \Omega\}  \quad \subset \R^{n+1},  \quad \quad i=1,\ldots,N,$$
which are in contact with one-another. In the most simplified form, the elastic membranes are modeled by harmonic graphs. The non-penetration condition implies the functions are ordered in the vertical direction, so we assume that
$$ u_1 \ge u_2 \ge \ldots \ge u_N.$$
Suppose all membranes coincide initially, say $u_1=u_2=\ldots =u_N=0$, and then we move them continuously by pulling apart their boundary data $t \varphi_i$ on $\p \Omega$ with $t \ge 0$, and $ \varphi_1>\varphi_2> \ldots > \varphi_n$. We consider the physical situation when the membranes do not separate strictly in the whole cylinder $\Omega \times \R$ instantaneously for all small $t>0$, but rather in a continuous way. In other words, for small $t$, the strict separation happens only near the boundary while the functions still coincide well in the interior of $\Omega$. This means the membranes {\it stick} to each other and we need to spend energy to physically separate them (say for example if the membranes are wet, or the elastic material has some adhesive properties).

\subsection{The model}   
In view of the discussion above, a natural energy functional associated to a system of adhesive elastic membranes is given by
\begin{equation}\label{JI} 
J_N(U, \Omega): = \int_\Omega (|\nabla U|^2 + W(U)) dx, \quad U:=(u_1,\ldots,u_N), \quad N\geq 2,
\end{equation}
with \begin{equation}\label{WI}|\nabla U|^2:= \sum_{i=1}^N |\nabla u_i|^2, \quad W(U)= \# \{ u_1,\ldots,u_N\},
\end{equation} and $J_N$ is defined over the class of admissible vector-valued functions
\begin{equation}\label{orderI} \mathcal A:= \left \{ U \in H^1(\Omega) | \quad u_1 \geq u_2 \geq \ldots \geq u_{N}  \right \}.\end{equation} 
Here $\Omega$ is a bounded domain in $\R^n$ with Lipschitz boundary, and the potential term $W(U)$ represents the cardinality of the set $\{ u_1,\ldots,u_N\}$ (that is the number of distinct elements in the set). Clearly $W(U)$ is minimized when all $u_i$'s coincide. On the other hand, the Dirichlet integral is minimized when each $u_i$ is harmonic, and by the maximum principle they belong to the admissible cone $\mathcal A$ provided the boundary data is in $\mathcal A$ as well. The presence of the potential term has the effect of collapsing some of these graphs that are close to each other in a certain region, and we expect an optimal configuration consisting of piecewise harmonic graphs with junctions.

The functional $J_N$ is lower semicontinuous and the existence of minimizers with given boundary data $\Phi \in \mathcal A$ on $\p \Omega$ follows easily from the direct method of the calculus of variations. We remark that $J_N$ is not convex, and the uniqueness of minimizers can fail. 

One of the interesting questions about minimizers $U$ of \eqref{J} regards the geometry of the graphs near junctions where the membranes separate. Precisely, there are possibly $N-1$ free boundaries originating from this minimization problem which are denoted by
$$\Gamma_i:= \partial \{u_i>u_{i+1}\} \cap \Omega, \quad i=1,\ldots, N-1.$$
The sets $\{u_i>u_{i+1}\}$ have no apriori constraints with respect to one-another, which means the $\Gamma_i$'s can intersect and cross each other and possibly have complicated geometries.
The functions $u_i$ are piecewise harmonic, that is they are harmonic in the interior of the regions carved out by the collections of the free boundaries. The problem is interesting even in dimension $n=1$ in which the $\Gamma_i$'s consist of points and the $u_i$'s are piecewise linear, and a certain balancing condition needs to hold at the junction points. The physical situation can be described as the equilibrium configurations of $N$-tapes that stick to one another, see Figure \ref{fig0}.

\begin{figure}[h] 
\includegraphics[width=0.4 \textwidth]{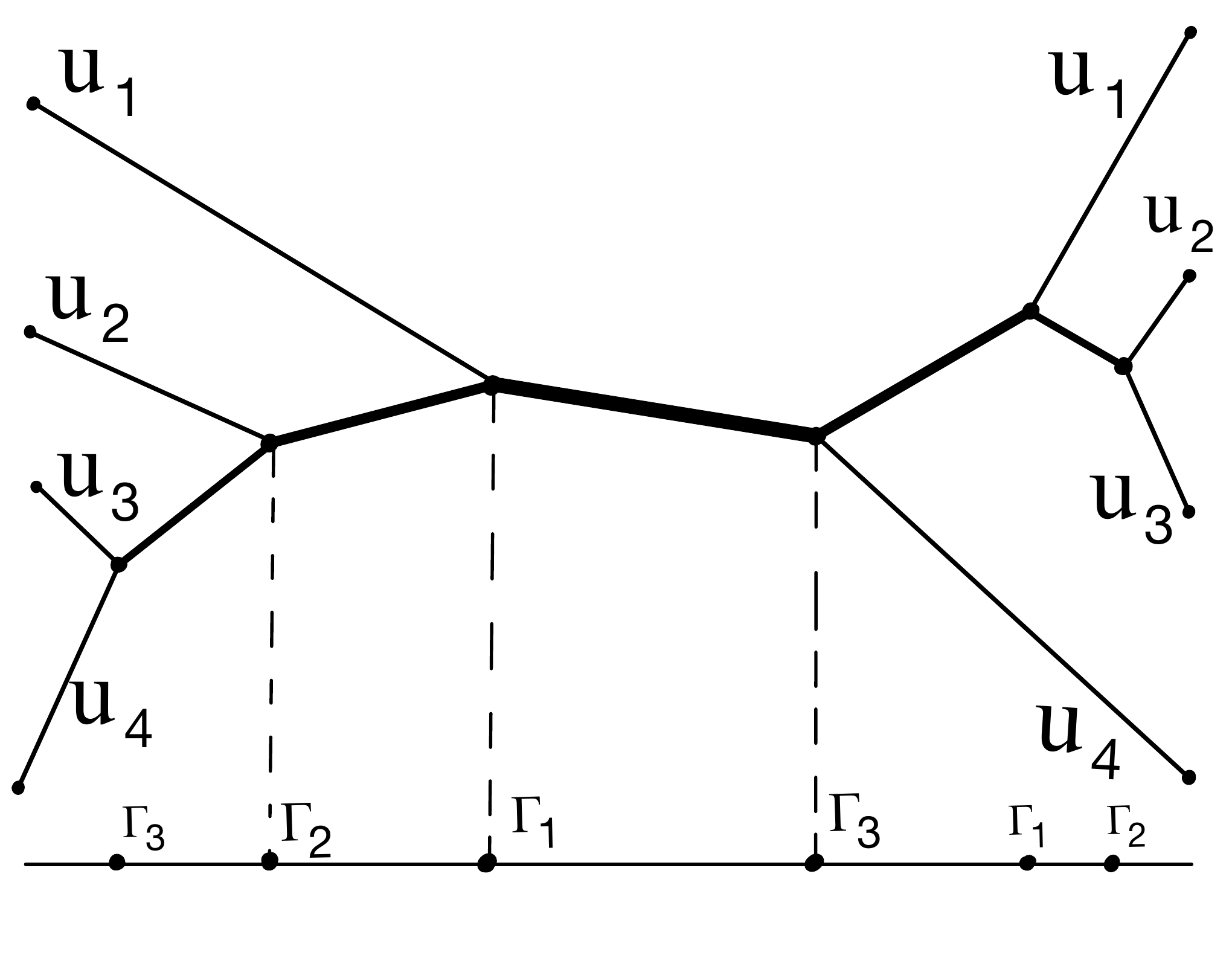}
\caption{A minimizer for $N=4$, $n=1$.} 
   \label{fig0}
\end{figure}  

A first observation about minimizers $U$ of \eqref{JI} is that the average of the $u_i$'s is harmonic in $\Omega$, see Lemma \ref{harmonic}. Moreover, minimizers remain invariant under the operation of adding the same harmonic function to each component of $U$. In view of this, one can reduce the problem to the $0$ average situation
$$ \sum u_i=0 \quad \mbox{in $\Omega$},$$ 
which means that we deal with a system involving only $N-1$ unknowns. Then, the case $N=2$ corresponds to the scalar minimization problem 
$$ \min_{u_1 \ge 0} \int_{\Omega} \left(2 |\nabla u_1|^2 + \chi_{\{u_1>0\}} \right)dx,$$
which is the classical one-phase free boundary problem introduced by Alt and Caffarelli in \cite{AC}. Here $\chi_E$ denotes the characteristic function of a set $E$. 

The one-phase problem appears in cavitation flows in fluid dynamics, flame propagation etc., and it has been extensively studied over the past four decades. We refer to the books of Caffarelli and Salsa \cite{CS} and Velichkov \cite{V} for the mathematical treatment of this problem. Concerning the free boundary regularity of $\Gamma_1$, it is an analytic hypersurface outside a closed singular set of dimension at most $n-5$ (see \cite{JS}), and there are examples of free boundaries with point singularities in dimension $n=7$, (see \cite{DJ}) . 

For $N\ge 3$ we subtract 1 from $W$, which does not affect the problem, and rewrite the potential term in the form
$$ W(U)= \sum_{i=1}^{N-1} \chi_{\{u_i>u_{i+1} \}}.$$
In this way, the model we propose can be viewed as a system of $N-1$ coupled one-phase free boundary problems that interact in the vertical direction.

\subsection{Some related works} There are several works in the literature that involve free boundaries and energy functionals for vector-valued functions similar to the one we propose in \eqref{JI}. For example Mazzoleni, Terracini, Velichkov \cite{MTV1,MTV2}, Caffarelli, Shahgholian, Yeressian \cite{CSY}, and De Silva, Tortone \cite{DT} considered the vectorial one-phase problem with $W(U)=\chi_{\{|U|>0\}}$ and unconstrained $U$, which is relevant in the study of cooperative systems of species or in optimization problems for spectral functions, see also \cite{KL1,KL2}. On the other hand, in \cite{ASUW} the authors studied a vectorial version of the obstacle problem by taking $W(U)=|U|$. In other cases motivated by strongly competitive systems or optimal partition problems, the components $u_i \ge 0$ have disjoint supports which means that the vector $U$ is restricted to the union of the nonnegative coordinate axes. For this situation we refer to \cite{CL} where Caffarelli and Lin investigated harmonic maps onto such singular spaces, see also \cite{TT}.

A minimization problem closely related to our model, which involves the constraint $U \in \mathcal A$ as in \eqref{orderI} and with the potential $W(U)=F \cdot U$, was introduced by Chipot and Vergara-Caffarelli in \cite{CV}. It describes the equilibrium configuration of $N$ elastic membranes interacting with each-other under the action of external forces $F$. More recently, the second author and H. Yu established the optimal regularity for this problem and studied the free boundary regularity in a series of papers \cite{SY1, SY2,SY3}.  While these results motivated the current work, the two models have quite different qualitative properties that can be seen from simple one-dimensional examples as well as from the theorems in the next section. 

Some of our results that involve the junction of $N=3$ membranes resemble the soap-films like configurations for minimal surfaces in two and three dimensions that were first studied by J. Taylor in \cite{T}. 
A crucial difference with respect to soap-films is that our energy model counts a common surface where more than two graphs coincide according to its multiplicity, and the collapsing phenomena is only due to the presence of the potential energy $W$.


\subsection{Main results}

In this section we state our main results. We recall that 
\begin{equation}\label{J} 
J_N(U, \Omega): = \int_\Omega (|\nabla U|^2 + W(U)) dx, \quad U=(u_1,\ldots,u_N), \quad N\geq 2,
\end{equation}
with \begin{equation}\label{W}|\nabla U|^2 = \sum_{i=1}^N |\nabla u_i|^2, \quad W(U):= \sum_{i=1}^{N-1} \chi_{\{u_i>u_{i+1} \}},
\end{equation} is defined over the class of admissible vector-valued functions
\begin{equation}\label{order} \mathcal A:= \left \{ U \in H^1(\Omega) | \quad u_1 \geq u_2 \geq \ldots \geq u_{N}  \right \}.\end{equation} 

Positive constants depending only on $N$ and the ambient dimension $n$ are called universal.
The first result gives the optimal regularity of minimizers.
\begin{thm}[Optimal regularity]\label{nonCI} Let $U$ minimize $J_N$ in $B_1$. Then, $U \in C^{0,1}$ and 
\begin{equation*}\|U\|_{C^{0,1}(B_{1/2})} \le C(1+ \|U\|_{L^2(B_1)}), \end{equation*} for $C>0$ universal.
\end{thm}

We study the regularity of the $N-1$ free boundaries
$$\Gamma_i:= \p \{u_i > u_{i+1} \}, \quad \quad i\in \{1,\ldots,N-1\},$$
by performing a blow-up analysis at junction points. After rescaling, we may assume that we are in the $0$-average situation, and the origin is a junction point where all membranes coincide
\begin{equation}\label{Res}
\sum u_i =0, \quad \quad 0 \in \p \{|U|>0\}.
\end{equation}
By employing Weiss monotonicity formula Proposition \ref{WMF}, and a non-degeneracy property of minimizers, we can deduce that blow-ups are one-homogenous.
\begin{thm} \label{Bup} Let $U$ minimize $J_N$ in $B_1$, and assume \eqref{Res} holds. Then
$$ \max_{B_r} |U| \ge c r,$$
and there exists a sequence of $r_j \to 0$ such that
$$ U_{r_j}(x):= \frac{U(r_j x)}{r_j} $$
converges uniformly on compact sets of $\R^n$ to a cone $\bar U$, i.e. a homogenous of degree one minimizer with $0 \in \p \{|\bar U| >0\}$.
\end{thm}

The minimizing cones in dimension $n=1$ must coincide on one side of $0$ and have two branches on the other side. Precisely for each $k \le N-1,$ we define $U_{0,k}$ to be the vector whose components are given by
$$ u_i=\left(\frac 1k - \frac 1N \right) x^+ \quad \quad \mbox{for $i \le k$}, \quad u_i=-\left(\frac {1}{N-k} - \frac 1N \right) x^+ \quad \quad \mbox{for $i > k$} .$$
\begin{prop}[1d cones]\label{P1dI}
The only minimizing cones in dimension $n=1$ that satisfy \eqref{Res} are given by $U_{0,k}$ (up to a reflection with respect to 0).
\end{prop}

The classification of 1d cones combined with an $\eps$-regularity theorem and a dimension reduction argument gives the following general partial regularity result. 
 \begin{thm}[Partial regularity]\label{prI}
Let $U$ be a minimizer in $B_1$. The free boundaries $\Gamma_i$ are analytic and disjoint from one another outside a closed set $\Sigma$ of singular points of Hausdorff dimension $n-2$, and
$$\mathcal H^{n-1}(\Gamma_i \cap B_{1/2}) \le C,$$
with $C$ universal.
\end{thm}

The remaining results focus on the intersection points of two distinct free boundaries which, by Theorem \ref{prI}, has dimension at most $n-2$. For this part we restrict to the simplest case that involves $N=3$ membranes. 

In dimension $n=2$ we define $V_0$ as the cone whose components are given by 
$$u_1(x):= \frac{1}{\sqrt 6}\max \left \{ x \cdot e_{ \frac \pi 6}, \, 2 x \cdot e_{- \frac \pi 6}, \, 0 \right \},\quad \quad e_\theta:=(\cos \theta, \sin \theta),$$
$$ u_3(x):=-u_1(x_1,-x_2), \quad \quad u_2 :=- u_1-u_3.$$
\begin{thm}\label{m2d} Let $N=3$. Then, up to rotations and reflections, $V_0$ is the unique minimizing cone in dimension $n=2$ which satisfies $0 \in \Gamma_1 \cap \Gamma_2$ and \eqref{Res}.

\end{thm}
The proof of Theorem \ref{m2d} is indirect, through an elimination process. It turns out that there are two possible cone candidates, $V_0$ and $V_s$, which are critical for the energy $J_N$. We show however that $V_s$ is not minimizing and also that the set of cones satisfying $0 \in \Gamma_1 \cap \Gamma_2$ and \eqref{Res} is nonempty, see Section \ref{Sec7}.

The next result gives the regularity of the free boundaries in two dimensions near an intersection point, see Figure \ref{fig1}.

\begin{thm}\label{T2dI}
Let $N=3$ and let $U$ be a minimizer of $J_N$ in $\Omega \subset \R^2$. Then $\Gamma_1$ and $\Gamma_2$ are piecewise $C^{1,\alpha}$ curves in a neighborhood of any intersection point $x_0 \in \Gamma_1 \cap \Gamma_2$, with $\alpha \sim 0.36$.
\end{thm}

\begin{figure}[h] 
\includegraphics[width=0.4 \textwidth]{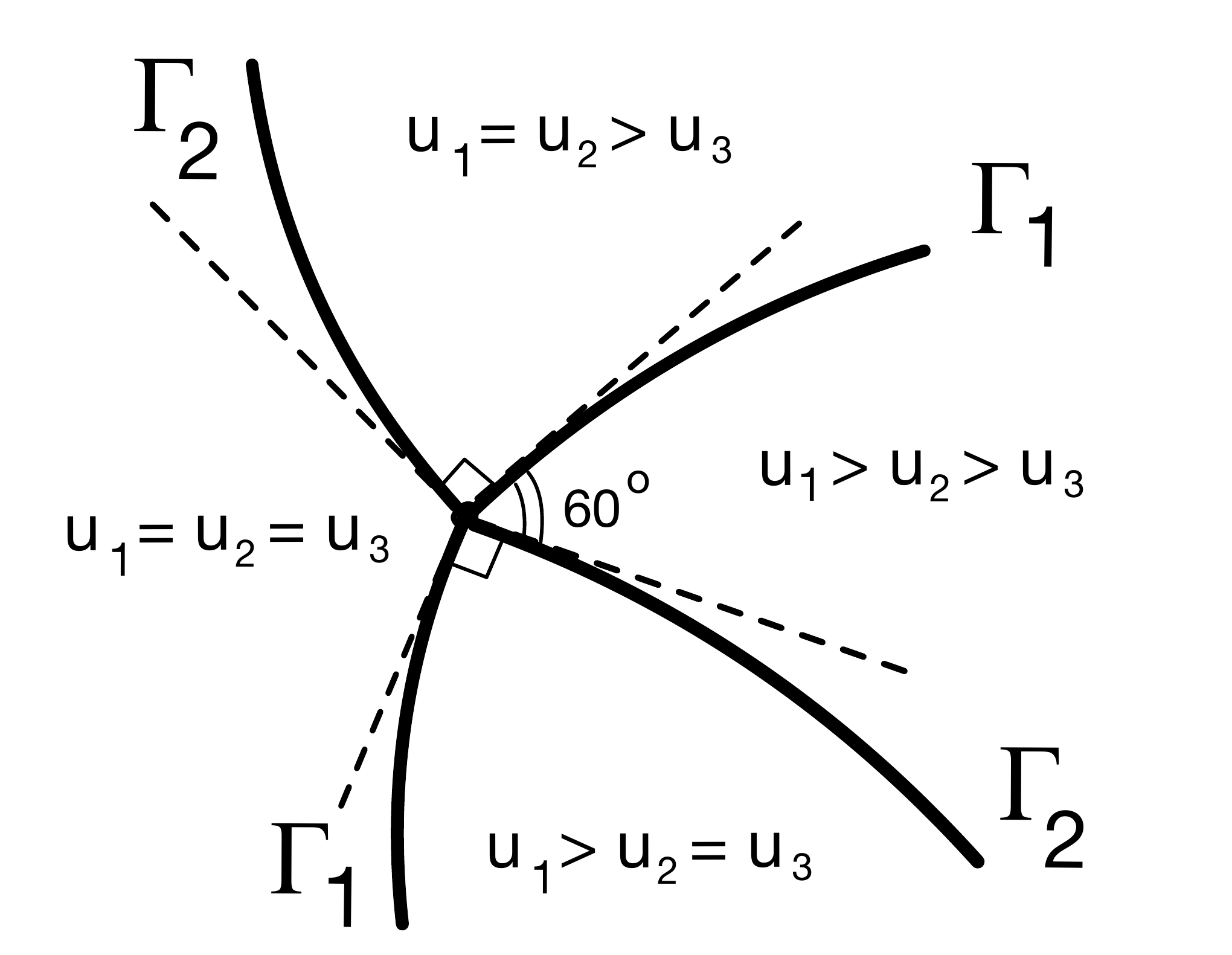}
\caption{ Free boundary geometry at an intersection point in $\mathbb R^2$} 
   \label{fig1}
\end{figure}

Theorem \ref{T2dI} can be extended to arbitrary dimensions. For this we define regular intersection points $$x_0 \in Reg(\Gamma_1 \cap \Gamma_2),$$ as those points for which there exists a blow-up cone at $x_0$ which is a rotation of a two-dimensional cone extended trivially in the remaining variables.

\begin{thm}\label{TndI}
$Reg(\Gamma_1 \cap \Gamma_2)$ is locally a $C^{1,\alpha}$-smooth submanifold of codimension two. Near such an intersection point, each of the free boundaries $\Gamma_1$ and $\Gamma_2$ consists of two piecewise $C^{1,\alpha}$ hypersurfaces which intersect on $Reg(\Gamma_1 \cap \Gamma_2)$ at a $120^ \circ$ angle.
\end{thm}

As a consequence we obtain a general partial regularity result for the free boundary.
We define $Reg(\Gamma_i)$ as the collection of points $x_0 \in \Gamma_i$ that have a blow-up profile which is a rotation of a one-dimensional cone extended trivially in the remaining variables.

\begin{thm}[Partial regularity]\label{TPRI} Let $N=3$ and $U$ be a minimizer of $J_N$ in $B_1$. Then
$$ \Gamma_1 \cup \Gamma_2 = Reg(\Gamma_1) \cup Reg(\Gamma_2)\cup Reg (\Gamma_1 \cap \Gamma_2) \cup \Sigma',$$
with $Reg(\Gamma_i)$ locally analytic hypersurface, $Reg (\Gamma_1 \cap \Gamma_2)$ locally a $C^{1,\alpha}$ submanifold of codimension two, and $\Sigma'$ a closed singular set of Hausdorff dimension $n-3$. 
\end{thm}

The study of regular intersection points leads to a transmission-type problem which appears in the linearization, which is new and interesting in its own, see \eqref{a1}-\eqref{a4} in Section \ref{Sec8}. The value of $\alpha$ in the theorems above is dictated by the spectrum of the linearized problem. The novelty is that the transmission condition does not occur along a hypersurface but it involves a region of full dimension where two functions interact. We refer the reader to Sections \ref{Sec8} and \ref{Sec9} for more details.

The paper is organized as follows. In Section \ref{Sec3} we collect some general facts about minimizers and prove the optimal regularity result Theorem \ref{nonCI}. In Section \ref{Sec4} we obtain the Weiss monotonicity formula and prove Theorem \ref{Bup}. We classify one-dimensional cones in Section \ref{Sec5} and then establish Theorem \ref{prI} in Section \ref{Sec6}. The last three sections are devoted to the study of two dimensional cones and regular intersection points for $N=3$ membranes.


\section{Lipschitz continuity and non-degeneracy of minimizers} \label{Sec3}

\subsection{Preliminaries.} Recall that, throughout this note, constants depending only on possibly $n,N$ are called universal. Also, whenever this does not create confusion, the dependence of $J_N$ on the domain is omitted.

We start by proving some basic facts about minimizers.

\begin{lem}[Lower semicontinuity]
If $U_m \to U$ in $L^2(\Omega)$, then
$$\liminf J_N(U_m,\Omega) \ge J_N(U,\Omega).$$
\end{lem}

\begin{proof}The Dirichlet energy is lower semicontinuous. Since $W: \R^N \to \R$ is lower semicontinuous, it follows the potential term in $J_N$ is lower semicontinuous as well. 

\end{proof}
As a consequence we obtain the existence of minimizers with boundary data in $H^1(\Omega)$. 

\begin{prop} Given $\Phi \in \mathcal A$, there exists a minimizer $U$ of $J_N$ in $\Omega$ with boundary data $\Phi$ on $\p \Omega$.

\end{prop}

Next we show that the average of the $u_i$'s is harmonic.

\begin{lem}\label{harmonic} If $U$ is a minimizer of $J_N$ in $\Omega$, then $\sum_{i=1}^N u_i$ is harmonic in $\Omega.$ \end{lem} \begin{proof} Let $\psi \in C_0^1(\Omega)$, and $\Psi=(\psi, \ldots, \psi)$. Then 
$$W(U)= W(U+\eps \Psi)$$ and 
$$\frac{J_N(U+\eps \Psi) - J_N(U)}{\eps} = \frac{1}{\eps}\int_{\Omega} ( |\nabla (U +\eps \Psi)|^2 - |\nabla U|^2) dx$$ $$= 2 \int_{\Omega} \nabla (\sum_{i=1}^n u_i) \cdot \nabla \psi dx + o(\eps),$$ from which our claim follows.\end{proof}

Similarly, the problem remains invariant with respect to addition of a harmonic function $\psi$ to each component, that is $$J_N(U + \psi(1,\ldots,1))= J_N(U) + C(\Psi,\Phi), \quad \Psi=\psi(1,\ldots, \psi)$$ for some constant $$C(\Psi,\Phi)=\int_\Omega |\nabla \Psi|^2 dx + \int_{\p \Omega}  \Phi \cdot \Psi_\nu d \sigma,$$ that depends only on $\psi$ and the boundary data $\Phi$ of $U$. Therefore, at times we may assume that $U$ satisfies,
\begin{equation}\label{zero}
\sum_{i=1}^N u_i=0.
\end{equation}

\begin{lem}\label{sub-super}Let $U$ minimize $J_N$ in $\Omega$. Then $u_1$ is subharmonic and $u_N$ is superharmonic in $\Omega$.
\end{lem}
\begin{proof}
We prove the first claim, the second one can be proved similarly. Let $v_1$ be the harmonic replacement of $u_1$ in $B_r \subset \Omega.$ We wish to prove that $u_1 \leq v_1$ in $B_r$. 
Indeed, let $$\tilde u_k = \min\{u_k, v_1\}, \quad \tilde U = (\tilde u_1,\ldots, \tilde u_N).$$
Clearly, 
$$W(U) \geq W(\tilde U),$$ hence by minimality
\begin{equation}\label{smaller}\int_{B_r} |\nabla U|^2 dx \leq \int_{B_r} |\nabla \tilde U|^2 dx.\end{equation}
On the other hand, since $$v_1 = u_1 \geq u_k, \quad \forall k=1, \ldots N, \quad \text{on $\p B_r$},$$ we have $$w_k:=(u_k-v_1)^+ \in H_0^1(B_r).$$ After integration by parts and using that $v_1$ is harmonic, we get
$$\int_{B_r} (|\nabla u_k|^2 - |\nabla \tilde u_k|^2) dx  = \int_{B_r}( |\nabla (v_1+w_k)|^2 - |\nabla v_1|^2)dx \geq 0,$$ with strict inequality unless $w_k \equiv 0$. Thus, in view of \eqref{smaller}, $w_k \equiv 0$ in $B_r$ for all $k$'s, which gives the desired claim that $u_1 \leq v_1$ in $B_r.$

Since this argument can be repeated for any ball included in $\Omega$, we conclude that $u_1$ is subharmonic.
\end{proof}

Next, we prove the following obvious yet useful fact.

\begin{lem}\label{induction} Let $U$ minimize $J_N$ in $\Omega$ and assume that $$u_{k} > u_{k+1} \quad \text{for some $k \in \{1, \ldots, N-1\}$}.$$ Then, 
$\underbar U := (u_1,\ldots, u_{k})$ minimizes $J_k$ in $\Omega$ and $\bar U := (u_{k+1}, \ldots, u_{N})$ minimizes $J_{N-k}$.
\end{lem}
\begin{proof}The proof is straightforward. Assume by contradiction that  $\underline{V}$ minimizes $J_k$ in $\Omega' \subset \subset \Omega$ with boundary data $\underline U$ and \begin{equation}\label{bar}J_k(\underline V) < J_k(\underline U).\end{equation} Let $\bar V$ minimize $J_{N-k}$ in $\Omega'$ with boundary data $\bar U$, and set
$$V= (\underline V, \bar V).$$ 
Since, by Lemma \ref{sub-super} $\underline{v}_k$  is superharmonic while  $\bar v_{k+1}$ is subharmonic and one in strictly on top of the other one on the boundary, we conclude that $\underline v_k \geq \bar v_{k+1}$ in $\Omega'$, hence $V$ is an admissible competitor. 
Furthermore,  in view of \eqref{bar} and our assumption,
$$J_N(V)  \leq J_k(\underline V)+ J_{N-k}(\bar V)+ |\Omega'| < J_k(\underline U)+ J_{N-k}(\bar U)+ |\Omega'|= J_N(U)$$ 
and we reach a contradiction.
\end{proof}

\subsection{Lipschitz Continuity.} We now turn to the proof of Lipschitz continuity. We start first by establishing that minimizers are H\"older continuous, by treating the potential term $W(U)$ as a perturbation. The proof is standard but we provide the details, for completeness.

After multiplying $U$ by $\delta^{1/2}$, we may assume that $U$ minimizes 
$$J_N^\delta(U):=\int_\Omega (|\nabla U|^2 + \delta W(U)) dx,$$
with $\delta>0$ small.

Without loss of generality, we assume that we deal with minimizers in $B_1$.
H\"older continuity follows with standard arguments from the next proposition. 

From now on, in the body of the proofs $c,C>0$ are universal constants possibly changing from step to step.

\begin{prop}\label{prop1} Let $U$ be a minimizer of $J_N^\delta$ in $B_1$. For every $\alpha \in (0,1),$ there exist $\delta_0, \rho>0$ depending on $\alpha$, such that, if $\delta \leq \delta_0$ and 
\begin{equation}\label{ass1}
\fint_{B_1} |\nabla U|^2 dx \leq 1,
\end{equation}
then
\begin{equation}
\rho^{2-2\alpha}\fint_{B_{\rho}} |\nabla U|^2 dx \leq 1.
\end{equation}
\end{prop}
\begin{proof} Let $V$ be the harmonic replacement of $U$ in $B_1$. Then, by the maximum principle $v_i \geq v_{i+1}, i=1,\ldots, N$, and $V$ is an admissible competitor for the minimization of $J_N^\delta$. Thus, 
\begin{align}\label{new}
\nonumber \int_{B_1} |\nabla (U-V)|^2 dx & \leq J_N^\delta(U) + \int_{B_1} (|\nabla V|^2 - 2 \nabla U \cdot \nabla V) dx\\ 
\nonumber & = J_N^\delta(U) - \int_{B_1} |\nabla V|^2 dx \\
&  \leq J_N^\delta(V) - \int_{B_1} |\nabla V|^2 dx \\
\nonumber  &=\int \delta W(V) dx \leq \delta N.
 \end{align} 
 Hence,
 \begin{equation}\label{1}
  \fint_{B_\rho} |\nabla (U-V)|^2 dx \leq \delta N \rho^{-n}.
 \end{equation}
 On the other hand, since $|\nabla V|^2$ is subharmonic in $B_1$ and by \eqref{ass1} 
$$\int_{B_1} |\nabla V|^2 dx \leq \int_{B_1} |\nabla U|^2 dx\leq C,$$ we conclude that
$$|\nabla V| \leq C, \quad \text{in $B_{1/2}$}.$$
In particular,
$$\fint_{B_\rho} |\nabla V|^2  dx\leq C_0,$$ for $C_0$ universal.
Combining this last inequality with \eqref{1} we conclude that,
$$\fint_{B_\rho} |\nabla U|^2 dx \leq 2 \delta N \rho^{-n} + 2 C_0 \leq  4 C_0,$$
as long as
$$\delta \leq \frac {C_0}{N} \rho^n.$$
Thus,
$$\rho^{2-2\alpha}\fint_{B_\rho} |\nabla U|^2 dx \leq 4C_0 \rho^{2-2\alpha} \leq 1,$$ by choosing $\rho$ such that,
$$\rho^{2-2\alpha} \leq (4C_0)^{-1}.$$
\end{proof}

The next corollary now follows via standard arguments.

\begin{cor}\label{holder} Let $U$ be a minimizer to $J_N$ in $B_1$. Then $U \in C_{loc}^{0,\alpha} (B_1)$, for any $\alpha \in (0,1)$ and $$[U]_{C^\alpha(B_{1/2})} \le C(1+\|\nabla U \|_{L^2(B_1)}),$$
with $C(\alpha,n,N)$.\end{cor}

Prior to proving Lipschitz continuity, we introduce the following definition.

\begin{defn}\label{conn}
We say that $U$ is disconnected in $B_\rho$ if there exists $k \in \{2,\ldots, N\}$ such that $u_k>u_{k-1}$ in $B_\rho$. Otherwise we say that $U$ is connected in $B_\rho$.
\end{defn}

Our main Lipschitz regularity result reads as follow.
\begin{lem}[Uniform Lipschitz estimate]\label{uniLip}Let $U$ minimize $J_N$ in $B_1$ and
assume that $\sum_{i=1}^N u_i=0$ and $U$ is connected in $B_{1/2}$. Then \begin{equation}\label{con1}\|U\|_{C^{0,1}(B_{1/2})} \le C,\end{equation} for $C>0$ universal.
\end{lem}

Then, the next theorem is an immediate corollary.
\begin{thm}\label{nonC} Let $U$ minimize $J_N$ in $B_1$. Then, $U \in C^{0,1}(B_{1/2})$ and 
\begin{equation}\label{in1}\|U\|_{C^{0,1}(B_{1/2})} \le C(1+ \|U\|_{L^2(B_1)}), \end{equation} for $C>0$ universal.
\end{thm}

We use Lemma \ref{uniLip} to prove Theorem \ref{nonC} by induction on $N$.

The case $N=1$ corresponds to harmonic functions and it is obvious. 

Next assume the claim holds for $j$ membranes, for any $j \leq N-1$. It suffices to prove \eqref{in1} only in the case of $N$ membranes with $\sum u_i=0$. 

Indeed, in the general case we let $h:=\sum_i u_i$ which is harmonic, and then $U= H + V$ with $H:= (h,\ldots, h)$ and $\sum_i v_i=0$. Since $H$ satisfies \eqref{in1} and
$$\|H+V\|_{L^2} = \|H\|_{L^2} + \|V\|_{L^2},$$  it suffices to show that $V$ satisfies \eqref{in1} as well. 

 If the $N$ membranes are connected in $B_{1/2}$ then the claim follows from Lemma \ref{uniLip}. If they are disconnected in $B_{1/2}$, then in view of Lemma \ref{induction}, by the induction hypothesis we have 
\begin{equation}\label{in2}\|U\|_{C^{0,1}(B_{1/4}(x_0))} \le C(1+ \|U\|_{L^2(B_{1/2})}), \quad \forall x_0 \in B_{1/4},\end{equation} which gives the desired claim.

\medskip

\textit{Proof of Lemma $\ref{uniLip}$}. 
We argue by induction. In the case $N=1$, $u_1 \equiv 0$ and the statement is trivial. Let us assume that \eqref{con1} holds when $U$ minimizes $J_k$ and $k<N.$ Then, by the proof above, also the estimate \eqref{in1} holds for minimizers of  $J_k$ and $k<N.$

 Let $V$ be the harmonic replacement of $U$ in $B_1$, and by the maximum principle $v_i \ge v_{i+1}$, and moreover  \begin{equation}\label{vi}\sum_{i=1}^N v_i=0.\end{equation} Then, arguing as in \eqref{new} in Proposition \ref{prop1}, and by Poincare's inequality we obtain
\begin{equation}\label{UV}\|U-V\|_{L^2(B_{1})},\|\nabla (U-V)\|_{L^2(B_1)} \le C.\end{equation}
Set
$$\mu:= \max_i \, \, \{v_i(0)\} = v_1(0).$$ 
By Harnack inequality,
$$ v_i -v_{i+1} \geq c (v_i -v_{i+1})(0) \quad \text{in $B_{7/8}$}, \quad \forall i=1,\ldots, N-1,$$ and in view of \eqref{vi} simple arithmetic gives that for some $k \in \{1,\ldots, N-1\}$,
\begin{equation}\label{k}v_k -v_{k+1} \geq c (v_k -v_{k+1})(0) \geq c \mu \quad   \text{in $B_{7/8}$}.\end{equation}
Furthermore, Harnack inequality for $v_1 \geq 0$ and $-v_N \geq 0$ also guarantees that
\begin{equation}\label{|V|}|V|= \max_i\{|v_i|\} = \max\{v_1, -v_N\} \leq C \mu \quad \text{in $B_{7/8}$}.\end{equation}

We now wish to show that $\mu$ is bounded by a universal constant. Indeed,
if $\mu \gg 1$ then, by the second inequality in \eqref{UV}, \eqref{|V|}, and the fact the $V$ is harmonic, we get $$ \|\nabla U\|_{L^2(B_{7/8})} \le C (\mu + 1) \le  C \mu.$$
Hence, by Corollary \ref{holder} and the fact that $V$ is harmonic and satisfies \eqref{|V|}, we get in $B_{3/4}$, $$[U]_{C^\alpha} \le C \mu, \quad [U-V]_{C^\alpha} \le C \mu.$$ Thus, if at $x_0\in B_{1/2}$, $|(U-V)(x_0)|=t$, then $|U-V| \ge t/2$ in a small neighborhood of size $r \sim (t/\mu)^{1/\alpha}$. By the first inequality in \eqref{UV} we get that $t \leq C \mu^{n/(2\alpha+n)}.$ 
Hence, 
$$ |U-V| \le C \mu^{1-\delta} \quad \text{in $B_{1/2}$},$$ 
for some $0<\delta<1.$
In view of \eqref{k}
$$u_k -u_{k+1} \geq c\mu -C \mu^{1-\delta}>0 \quad \mbox{in} \quad B_{1/2}, $$ 
and we contradict that $U$ is connected in $B_{1/2}$. Thus $\mu \le C$ as desired, and it follows from the estimates above that
\begin{equation}\label{h1}\|U\|_{H^1(B_{1/2})} \le C.\end{equation}
This proves the statement for the $H^1$ norm instead of the $C^{0,1}$ norm.

 In order to conclude our proof, it suffices to obtain a scale invariant version of this $H^1$ estimate, i.e. 
$$ \fint_{B_{r_k}} |\nabla U|^2 \leq C, \quad r_k=2^{-k}, \quad \forall k\geq 1.$$ In terms of the rescalings $$\tilde U_r(x):=U(rx)/r, \quad x \in B_1,$$ this is equivalent to
$$\fint_{B_{1}} |\nabla \tilde U_{r_k}|^2 \leq C, \quad \forall k \geq 1.$$

Let $m \geq 1$ be the first value for which $U$ is connected in $B_{r_m}$ but it is not connected in $B_{r_{m+1}}.$ If no such $m$ exists, then \eqref{h1} holds for all $\tilde U_{r_k}$ and the desired statement follows. Otherwise, in view of \eqref{h1},
\begin{equation}\label{usubj}
\|U_{r_k}\|_{H^1(B_{1/2})} \leq C, \quad \forall k \leq m.
\end{equation} Since $U$ is disconnected in $B_{r_{m+1}}$, we can use the induction hypothesis and conclude that
$$\|U_{r_m}\|_{C^{0,1}(B_{1/4})} \leq C(1+ \|U_{r_m}\|_{L^2(B_{1/2})}) \leq C.$$ Then \eqref{usubj} holds for all $k \ge m$ and we reached the conclusion.
\qed

\smallskip

Having established the continuity results above, we can obtain the following compactness property of minimizers.

\begin{prop}[Compactness of minimizers] If $U_m$ is a sequence of minimizers in $B_1$ and $U_m \to \bar U$ uniformly on compact sets, then $\bar U$ is a minimizer and 
$$ J_N(U_m,B_r) \to J_N(\bar U, B_r) \quad \quad \forall r \in (0,1).$$
.\end{prop}

\begin{proof}
Since the $U_m$'s are continuous, then $\bar U$ is continuous and the $U_m$'s  are uniformly bounded on compact sets. In view of the Lipschitz continuity Theorem \ref{nonC}, we conclude that the $U_m$'s and $\bar U$ are uniformly Lipschitz on compacts.

Let $V$ be a competitor which agrees with $\bar U$ in $B_1 \setminus B_{r}$. Let $0\leq \varphi \leq 1$ be a radial cutoff function which is 1 in $B_{r}$ and vanishes outside $B_{r+\delta}$.Then 
$$ V_m:= \varphi \,  V + (1- \varphi)U_m $$
agrees with $U_m$ in $B_1 \setminus B_{r+\delta}$, and it is an admissible competitor as well. Thus, minimality implies that
\begin{equation}\label{up}J_N(U_m,B_{r+\delta}) \le J_N(V_m, B_{r+\delta}).\end{equation}
Since
$$ \nabla V_m = (V-U_m) \nabla \varphi + \varphi \nabla V + (1-\varphi) \nabla U_m,$$
$$ |\nabla U_m|, |\nabla {\bar U}| \le K , \quad \quad W(U_m), W(V_m) \le N, \quad |\nabla \varphi| \le C \delta^{-1},$$
we have that in the annulus $\mathcal A_\delta:=B_{r+\delta}\setminus B_r$,
$$ J_N(V_m, \mathcal A_\delta ) - J_N(U_m,\mathcal A_\delta) \le C(1+K^2)|\mathcal A_\delta| + C \delta^{-1} \|U_m-\bar U\|^2_{L^\infty(\mathcal A_\delta)}.$$
The right hand side can be made arbitrarily small by first choosing $\delta$ small depending on $K$ and them $m$ large. On the other hand, in view of \eqref{up} we have
$$J_N(U_m, B_r) \leq J_N(V_m, \mathcal A_\delta ) - J_N(U_m,\mathcal A_\delta) + J_N(V,B_r).$$
In conclusion
$$\limsup J_N(U_m,B_r) \le  J_N(V, B_r),$$
while the lower semicontinuity of the energy gives
$$ J_N(\bar U,B_r) \le \liminf J_N(U_m,B_r) .$$
\end{proof}

\subsection{Non-degeneracy}

\begin{lem}[Caccioppoli's estimate]\label{CE}Let $U$ minimize $J_N$ in $B_1.$ Then,
$$ \int_{B_{1/2}} (|\nabla U|^2 + W(U)) \, dx \le C \int_{B_1} |U|^2 dx,$$ for $C>0$
 universal.\end{lem}
\begin{proof}Let $0 \leq \phi \leq 1$ be a smooth bump function, which equals 1 in $B_{1/2}$ and $0$ outside of $B_{3/4}$.
Set, 
$$\tilde u_k:= u_k(1-\phi), \quad k=1,\ldots, N.$$ Then $\tilde U=(\tilde u_1,\ldots, \tilde u_N)$ is an admissible competitor and by minimality,
$$\int_{B_1}( |\nabla U|^2+ W(U))\; dx \leq \int_{B_1} (|\nabla \tilde U|^2+ W(\tilde U))\; dx,$$
hence 
$$\text{$W(\tilde U) \leq W(U)$ and $W(\tilde U)=0$ in $B_{1/2}$,}$$ implying that 
\begin{equation}\label{Ca}\int_{B_1}(|\nabla U|^2  + W(U)\chi_{B_{1/2}})\; dx\leq \int_{B_1}|\nabla \tilde U|^2\; dx.\end{equation} Moreover,
\begin{align*}
\int_{B_1}|\nabla \tilde U|^2\; dx & = \int_{B_1}((1-\phi)^2|\nabla U|^2 +  |U|^2 |\nabla \phi|^2 - (1-\phi)\nabla |U|^2 \cdot \nabla \phi) \, \, \; dx \\
& \leq \int_{B_1\setminus B_{1/2}}|\nabla U|^2 \; dx + \int_{B_1}|U|^2(1-\phi)\Delta \phi\; dx,
\end{align*} where in the last inequality we integrated by parts. Using this in \eqref{Ca} we get,
$$\int_{B_{1/2}} (|\nabla U|^2  + W(U)) \; dx \leq C\int_{B_1} |U|^2\;dx$$ with $C=C(n)$ as desired.
\end{proof}

A direct consequence of the Caccioppoli's estimate is that $|U|$ cannot be small and strictly positive in $B_1$.

\begin{cor}\label{CorC1}
If $\sum_{i=1}^N u_i=0$ and $B_r(x_0) \subset \{|U| >0\}$ then 
$$ \max_{B_{r}(x_0)} |U| \ge c r,$$
for some $c>0$ universal.
\end{cor}

Indeed, by scaling we may assume that $x_0=0$, $r=1$. Then $|U|>0$ implies $W(U)\ge 1$ in $B_1$ and the conclusion follows by the estimate in Lemma \ref{CE}.

We state a stronger version of non-degeneracy for $|U|$ than in Corollary \ref{CorC1}.
\begin{lem}[Non-degeneracy]\label{NDeg}
Let $U$ be a minimizer in $B_1$. If $\sum_{i=1}^N u_i=0$ and $0 \in \p \{|U|>0\}$ then $$\max _{B_r} |U| \ge c r, \quad \quad \forall r <1.$$
\end{lem}

\begin{proof}
We use the previous corollary and a standard ball sequence construction for the classical one-phase problem. Notice that $\sum_{i=1}^N u_i=0$ and $u_1 \ge ..\ge u_N$ implies that 
$$c_0 |U| \le u_1 \le |U|,$$
with $c_0>0$ depening on $N$ and $n$, and recall that by Lemma \ref{sub-super} $u_1$ is a subharmonic function.

Assume that $r=1$. We construct inductively a sequence of points $x_k \in \{|U|>0\}$ such that $x_0$ is sufficiently close to the origin and

1) $u_1(x_{k+1})   \ge (1+ c_1) u_1(x_k)$

2) $$ c_2 r_k \le u_1(x_k) \le C_2 r_k, \quad \quad  |x_{k+1}-x_k| \le 4 r_k, \quad \quad r_k:=d(x_k, \{|U|=0\}).$$ 
The inequality $u_1(x_k) \le |U(x_k)| \le C_2 r_k$ is a consequence of Lemma \ref{uniLip}.

Let's assume that $x_1$,..,$x_k$ are constructed, and let $$y_k \in \p \{|U|=0\} \cap \p B_{r_k}(x_k).$$ Since $u_1 \ge 0$ is subharmonic in $B_{r_k}(x_k)$, $u_1(x_k) \ge c_2 r_k$ and, by Lemma \ref{uniLip}, $$u_1 (x) \le |U(x)| \le C_2 |x-y_k|,$$ 
it follows from the mean value inequality that there exists $z_k \in B_{r_k}(x_k)$ such that $u_1(z_k)   \ge (1+ c_1) u_1(x_k)$ for some $c_1>0$ depending on $c_2$, $C_2$ and $n$.
We pick $x_{k+1} \in B_{t_k }(z_k)$ to be the point where the maximum value of $u_1$ occurs, where $t_k$ represents the distance from $z_k$ to the set $\{|U|=0\}$. Then 
\begin{align*}
u_1(x_{k+1}) & = \max_{B_{t_k }(z_k)} u_1 \\
& \ge c_0 \, \max_{B_{t_k }(z_k)} |U| \\
& \ge c_0 \, c \,  t_k \\
& \ge c_2 \, d(x_{k+1},\{|U|=0\}),
\end{align*} 
where in the second inequality we have used Corollary \ref{CorC1}, and we chose $c_2$ small depending on $c_0$ and $c$. This proves the existence of the sequence $x_k$.

Then we have $$(1+ c_1) u_1(x_k) \le u_1(x_{k+1}) \le C_2 r_{k+1} \le C_3 r_k \le C_4 u_1(x_k),$$ 
hence $u_1(x_k)$ is a sequence that grows at a geometric rate. The conclusion of the lemma follows easily since $u_1(x_k) \sim r_k$.

\end{proof}

Next we show the weak non-degeneracy for the difference of two consecutive membranes. It is not clear whether the strong non-degeneracy holds for two consecutive membranes as well. This lemma will not be used in the remaining of the paper. 

\begin{lem}If $B_r \subset \{u_k > u_{k+1}\}$ is tangent to $\Gamma_k$ then $(u_k - u_{k+1}) (0) \ge c r$.
\end{lem}
\begin{proof}
We only sketch the proof. After rescaling, we can assume that $U$ minimizes $J_N$ in $B_1$, $\sum u_i =0$, and 
\begin{equation}\label{pos}u_k - u_{k+1} >0 \quad \text{in $B_1$}.\end{equation} We wish to show that
$$(u_k - u_{k+1})(0) \geq c_0,$$ for some $c_0$ universal. 

We may also assume that $\bar U=\{u_1,..,u_k\}$ and $\underline U=\{u_{k+1},..,u_N\}$ are connected in $B_{\delta}$ for some $\delta>0$ small to be chosen later (see Definition \ref{conn}). Otherwise we can reduce the problem to one in $B_\delta$ with a strictly smaller number of membranes $N$. 
Let $\bar h$, $\underline h$ denote the averages of the membranes in $\bar U$, respectively $\underline U$.

Then $\bar h$, $\underline h$ are harmonic in $B_1$ and by Lemma \ref{uniLip} we know that 
$$ \|u_i - \bar h\|_{C^{0,1}(B_{1/2})} \le C, \quad \quad i \in \{1,..,k\}.$$
Since $\bar U$ is connected in $B_\delta$ we find that
$$ \|u_i - \bar h\|_{L^\infty(B_\delta)} \le C \delta,$$
hence 
$$ u_k \le \bar h, \quad \quad \bar h(0) \le u_k(0) + C \delta.$$
Similarly we obtain
$$ u_{k+1} \ge \underline h, \quad \quad \underline h(0) \ge u_{k+1}(0) - C \delta.$$
Assume by contradiction that 
\begin{equation}\label{con15}(u_k - u_{k+1})(0) \leq \delta. \end{equation} Since $\bar h > \underline h$, by Harnack inequality we find 
$$ \bar h - \underline h \le C \delta \quad \mbox{in $B_{1/2}$.}$$
This implies that $u_1 \le u_N + C \delta$ in $B_{1/2}$ and then by Corollary \ref{CorC1} it follows that all membranes coincide in $B_{1/4}$ a contradiction.

\end{proof}


\section{Monotonicity formula}\label{Sec4}

In this section we present the Weiss type monotonicity formula associated to our energy functional $J_N$ and prove Theorem \ref{Bup}. This is a standard tool necessary for the analysis of the partial regularity of free boundaries arising in a minimization problem, see for example \cite{W}.

\begin{prop}[Weiss monotonicity formula] \label{WMF}Assume that $U$ is a critical point to $J_N$ in $B_1$ with respect to continuous deformations in $\R^{n+1}$. Then for $r<1$,
$$\Phi(r):=r^{-n} \int_{B_r} \left(|\nabla U|^2+W(U)\right)dx - r^{-n-1} \int_{\p B_r} |U|^2 d \sigma, $$
is monotone increasing, and
$$\Phi'(r) = 2 \, r^{1-n} \int_{\p B_r}(r^{-1}U- U_\nu)^2   d \sigma \, \, \, \ge 0.$$
The $\Phi$ functional is constant in $r$ if and only if $U$ is ``a cone"  i.e. homogenous of degree one, and then this constant (the energy of $U$) is given by
$$\Phi(U)=\int_{B_1} W(U) dx.$$

\end{prop}

First we specify the notion of $U$ to be critical with respect to continuous deformations in $\R^{n+1}$. Given a smooth diffeomorphism $\Phi: \R^{n+1} \to \R^{n+1}$ with compact support in $\Omega \times \R$, we deform the $N$-graphs of $U$ by the map $$X \mapsto X + t \Phi(X), \quad \quad \mbox{$t$ small,} \quad X \in \R^{n+1},$$
onto the graphs of another admissible function that we denote by $U_t$. Then we say that $U$ is critical if $$\frac{d}{dt} J_N(U_t,\Omega) \, |_{t=0} =0. $$ 
If $\Phi$ is independent of the $x_{n+1}$ variable then $U$ becomes critical for $J_N$ with respect to the standard domain deformations in $\R^n$.

We start with the following lemma, in which we compute the first variation for critical points with respect to domain deformations.

\begin{lem}[First variation]\label{FV}
If $U$ is a critical point of $J_N$ with respect to domain variations then
\begin{equation}\label{fv}
\int_{B_1} \left(-2 \, \psi^k_l U_k \cdot U_l + (|\nabla U|^2 + W(U)) div \, \Psi \right)\, \, dx =0,
\end{equation}
for any Lipschitz map $\Psi: \R^n \to \R^n$ with compact support in $B_1$. 

In particular, for a.e. $r \in (0,1)$
\begin{equation}\label{div}
\frac{d}{dr} (r^{-n} J_N(U,B_r)) = 2 r^{-n} \int_{\p B_r} |U_\nu| ^2 d \sigma -2 \, r^{-n-1} \int_{B_r}|\nabla U|^2 dx. 
\end{equation}
\end{lem}

\begin{proof}
Let $\Psi: \R^n \to \R^n$ be a a Lipschitz map with compact support in $B_1$, and we consider the domain deformation
$$x \mapsto x + \eps  \, \Psi(x) =: y.$$
Then
$$ D_x y = I + \eps \, D \Psi, \quad \quad D_y x = (D_x y)^{-1}= I - \eps D \Psi + O(\eps^2), $$
$$ dy=1+ \eps \,  div \, \Psi + O(\eps^2),$$
and
\begin{align*}
&J_N(U(y), B_1) = \int_{B_1} |\nabla_y U|^2 + W(U(y)) \, dy \\
&=\int_{B_1} \left(\nabla_x U (I - \eps \, D \Psi)(I - \eps D \Psi)^T (\nabla_x U)^T + W(U)\right) (1+ \eps \, div \, \Psi) \, \, dx + O(\eps^2)\\
& = J_N(U,B_1) + \eps \int_{B_1} \left(-2 \psi^k_l U_k \cdot U_l + (|\nabla U|^2 + W(U)) div \, \Psi \right)\, \, dx + O(\eps^2).
\end{align*}
Here $O(\eps^2)$ depends on $\|U\|^2_{H^1}$ and $\|D\Psi\|_{L ^\infty}$.

We take in \eqref{fv} $$\Psi(x)= \psi_\delta(|x|) x $$ with $\psi_\delta$ a cutoff function which is 1 in $[0,r-\delta]$ and $0$ in $[r , \infty)$, and we let $\delta \to 0$ and obtain
$$\int_{B_r} \left((n-2)|\nabla U|^2 + n W(U)\right) dx = r \int_{\p B_r} \left(|\nabla U|^2 + W (U) - 2 |U_\nu|^2\right) \, d \sigma,$$
which gives \eqref{div}.
\end{proof}

Let us consider continuous deformations of the graph of $U$ in the vertical direction
$$ (x, U) \mapsto (x, U + t \varphi(x) U)=: (x,U_t(x))$$
with $\varphi$ a smooth function with compact support. Notice that these deformations are admissible for all $t$ close to $0$, and in addition
$$ W(U)=W(U_t).$$ 
Then, if $U$ is critical for $J_N$ with respect to the family $U_t$ we find
\begin{equation}\label{div2}
0 = \int_{B_1} \nabla U \cdot \nabla (\varphi U) dx,
\end{equation}
which means that $$U \cdot \triangle U =0 $$ holds in the distribution sense. We take in \eqref{div2} $\psi(x)=\psi_\delta(|x|)$ as in the proof above and obtain that
\begin{equation}\label{div3}
\int_{B_r} |\nabla U|^2 dx = \int_{\p B_r} U \cdot U_\nu \, d\sigma \quad \quad \mbox{for a.e. $r$.}
\end{equation}

\begin{proof}[Proof of Proposition \ref{WMF}]

We compute
$$ \frac{d}{dr} \left(r^{-n-1} \int_{\p B_r} |U|^2 d \sigma \right) = 2 r^ {-n-1} \int _{\p B_r}  \left(U \cdot U_\nu - |U|^2\right) \, d \sigma,$$
which together with \eqref{div} and \eqref{div3} give the formula for $\Phi'(r)$.

If $U$ is a cone, then $\Phi(r)$ is a constant which we denote by $\Phi(U)$, and since
$$ \int_{B_1} |\nabla U|^2 dx=\int_{\p B_1} U \cdot U_\nu d \sigma = \int_{\p B_1}| U|^2 d \sigma,$$
we find
$$ \Phi(U)= \int_{B_1} W(U) dx.$$ 
\end{proof}

Having established Weiss' formula, we can now introduce the notion of blow-up cones. Let $U$ be a minimizer in $\Omega$ and assume $U(x_0)=0$. The Lipschitz continuity Theorem \ref{nonC} shows that $\Phi(r)$ is bounded below, hence it has a limit as $r \to 0$. For a sequence of $r_m \to 0$, consider the rescalings,
$$U_{r_m}(x):= \frac{U(x_0+r_mx)}{r_m}.$$ Then according to Theorem \ref{nonC}, up to extracting a subsequence, $U_m \to U_0$ uniformly on compact sets of $\R^n$. The limit $U_0$ is called a blow-up limit, and by our compactness theorem, $U_0$ is a minimizer.  It easily follows from Weiss' formula, that $U_0$ is homogeneous of degree 1.

This discussion and the non-degeneracy Lemma \ref{NDeg} gives Theorem \ref{Bup}.


\section{One dimensional minimizers}\label{Sec5}

We show that the minimizing cones to $J_N$ in dimension $n=1$ are given by $$u_1=u_2=..=u_k \ge u_{k+1}=..=u_N, \quad \quad k=1,\ldots, N-1,$$
 with $u_k=u_{k+1}$ for $x\leq 0$ and $u_k >u_{k+1}$ for $x>0$. More precisely, for $k \le N-1,$ define $U_{0,k}$ as
$$ u_i=\left(\frac 1k - \frac 1N \right) x^+ \quad \quad \mbox{for $i \le k$}, \quad u_i=-\left(\frac {1}{N-k} - \frac 1N \right) x^+ \quad \quad \mbox{for $i > k$} .$$
Notice that $$ |\nabla U_{0,k}|^2=1=W(U_{0,k}), \quad x>0,$$
and $$ \Phi(U_{0,k})=1,$$ where we are using the notation $$\Phi(U):= \int_{-1}^1 W(U) \;dx.$$
\begin{prop}\label{P1d}
The only minimizing cones $U$ in dimension $n=1$ are given by $U(x)=L(x) + U_{0,k}(x)$, with $L(x)= (a\cdot x, \ldots, a\cdot x), a \in \R^n$ (up to a reflection with respect to 0). If $\sum u_i = 0$ then $U=U_{0,k}.$
\end{prop}

Toward the proof of Proposition \ref{P1d}, we establish the next lemmas.

\begin{lem}[Inwards perturbations]\label{L1da}
Let $U$ be a minimizer to $J_N$ in an interval around $0$ in dimension $n=1$. Assume that a number of membranes $u_i$ with indeces $$i \in I=\{ j+1,..,k \},$$ 
coincide at $0$, and on a small interval $(0,\delta)$ to the right of $0$ 

a) $u_i$ are linear, 

b) the family $u_i$, $i \in I$, is strictly separated from the remaining membranes i.e. $u_{j-1}>u_j$, $u_k > u_{k+1}$. 

Then, in $(0,\delta)$ the slopes satisfy the inequality
$$ \sum_{i \in I} |\nabla (u_i - u_I)|^2 \ge W (\{u_j,..,u_k \}), \quad \quad u_I:= \frac{1}{k-j} \sum_{i\in I} u_i.$$
\end{lem}

\begin{proof}
After subtracting a linear function and after a dilation we reduce to the case $$\delta =1, \quad u_I=0.$$ Also, after relabeling the membranes we may assume that $$I=\{1,2,.., N\},$$ 
and $U$ minimizes the energy in $[0,1]$ among all admissible competitors $V$ with the same boundary data as $U$ in $[0,1]$ and with $v_1 \le u_1$ and $v_N \ge u_N$.    

 We pick 
$$ V_\eps (x):=\left \{ 
\begin{array}{l}
U(\frac {x-\eps}{1-\eps}) \quad \mbox{if} \quad x \in [\eps,1] \\

\

\\

0  \quad \mbox{if} \quad x \in [0,\eps],

\end{array}
\right .
$$
and find that
$$ J_N(V_\eps)= \int_0^1\left( \frac{1}{1-\eps} |\nabla U|^2 + (1-\eps) W(U)\right) dx. $$
The minimality implies
$$ 0 \le \frac{d}{d \eps} J_N(V_\eps)|_{\eps=0} = \int_0^1 (|\nabla U|^2 - W(U)) dx,$$
which gives the desired conclusion $$ |\nabla U|^2 \ge W(U).$$

\end{proof}

\begin{rem}\label{R1d}
If in addition to the hypotheses of Lemma \ref{L1da} we assume that the $u_i$ with $i \in I$ coincide in $[-\delta,0]$, and the separation with respect to the remaining membranes holds in $[-\delta,\delta]$, then the inequality becomes an equality:
 \begin{equation}\label{EL1d}
 \sum_{i \in I} |\nabla (u_i - u_I)|^2 = W (\{u_j,..,u_k \}), \quad \quad u_I:= \frac{1}{k-j} \sum_{i\in I} u_i.
 \end{equation}
 Indeed, in the proof above we may replace $\eps$ by $-\eps$, and the corresponding functions are still admissible.
 
 This equality can be interpreted as the Euler-Lagrange equation at the branching points, which expresses the equipartition of energy.
\end{rem}

\begin{lem}\label{L1db}Let $U$ be a minimizer to $J_N$ in $(-\delta,\delta)$ in dimension $n=1$ with
$$u_1=u_2=..=u_k, \quad  \mbox{for some $k \ge 2$.}$$
Assume that

a) $u_1$ is linear in $[0,\delta)$ and $(-\delta,0]$, and let $a^+$, respectively $a^-$ denote its slopes;

b) $u_k > u_{k+1}$ in $(0,\delta)$.

Then

$$ 0 \le a^+-a^- \le \frac{1}{\sqrt k}.$$

\end{lem}

\begin{proof}
As above, after subtracting a linear function we may assume that $\delta=1$ and $a^-=0$. Then $a^+ \ge 0$ since $u_1$ is subharmonic,  according to Proposition \ref{sub-super}. 

The top $k$ membranes vanish to the left of $0$, and we can use as an admissible competitor for these functions the same one used in the proof of Lemma \ref{L1da} with $\eps$ replaced by $-\eps$. Precisely, let 
$$ v_{\eps,i}:=\frac{1}{1+\eps} u_i(x+\eps) \chi_{[-\eps,1]} \quad \mbox{if $i \le k$,}\quad \mbox{and} \quad v_{\eps,i} := u_i \quad \mbox{if $i > k$},$$
and notice that
\begin{align*}
J_N(V_\eps) - J_N(U) & = \int_{-1}^1 \left(\sum_{i=1}^k (|\nabla v_{\eps,i}|^2 -  |\nabla u_i|^2) + \chi_{\{v_{\eps,k} > u_{k+1}\} \cap [-\eps,0]} \right)dx\\
& \le  \eps - \frac{\eps}{1+\eps} \int_0^1 \sum_{i=1}^k |\nabla u_i|^2 dx, 
\end{align*}
which, as $\eps \to 0$, gives the desired inequality $$ k (a^+)^2 \le 1.$$
\end{proof}

We are now ready to provide the proof of Proposition \ref{P1d}.

\begin{proof}[Proof of Proposition $\ref{P1d}$]
Let $U$ be a nonzero minimizing cone of average 0 in dimension $n=1$ and let $$a_1^+ > a_2^+>...>a_l^+$$ be the slopes of the membranes on $[0, \infty)$ and denote by $k_1^+,..,k_l^+$ their multiplicities that is, $k^+_j$ is the number of membranes with slope $a^+_j$. Since $U$ has average $0$ we have $a_1^+  \ge 0 \ge a_l^+$. Similarly on $(-\infty,0]$ we have slopes $-a_1^-,-a_2^-,..,-a_m^-$ with $$a_1^->...>a_m^-, \quad \quad a_1^- \ge 0 \ge a_m^-,$$ with multiplicities $k_1^-,..,k_m^-$. We apply Lemma \ref{L1da} for the collection of membranes corresponding to 2 consecutive slopes $a_i$ and $a_{i+1}$ (here we drop the $+$ or $-$ superscript for simplicity of notation) and obtain 
$$ (a_i -a_{i+1})^2 \left ( k_i \frac{k_{i+1}^2}{(k_i+k_{i+1})^2} +  k_{i+1} \frac{k_{i+1}^2}{(k_i+k_{i+1})^2}\right) \ge 1,$$
or
$$ (a_i - a_{i+1})^2 \ge \frac{1}{k_i} + \frac{1}{k_{i+1}}.$$
Next assume that 
\begin{equation}\label{999}
k_1^+ \le k_1^-.
\end{equation}
 We can apply Lemma \ref{L1db} to the top $k_1^+$ membranes and obtain
$$ a_1^+ + a_1^- \le \sqrt {\frac{1}{k_1^+}}.$$ 
If $l=1$ then, $k_1^+ =N$ and therefore $k_1^-=N$ as well. This means $l=m=1$ hence $U \equiv 0$ which contradicts our hypothesis. 
In conclusion, $l \ge 2$ and then the inequalities above (with $i=1$) imply
\begin{equation}\label{1000}
 a_2^+ < - a_1^- \le 0,
 \end{equation}
and since $a_l^+ \le a_2^+$, we find
\begin{equation}\label{1001}
 a_1^- + a_l^+ <0.
 \end{equation}
We claim that 
\begin{equation}\label{1002}
k_l^+ < k_m^-.
 \end{equation}

Indeed, otherwise $k_l^+ \ge k_m^-$ and we can argue as in \eqref{1001} for the membranes in the reversed order (notice that $(u_1,..,u_N)$ is a minimizer if and only if $(-u_N,..,-u_1)$ is a minimizer). We obtain 
$$-a_1^-  -a_l^+ <0,$$
and contradict \eqref{1001}, hence the claim \eqref{1002} is proved.

If $l \ge 3$ then we may apply \eqref{1000} for the reversing order (see \eqref{1002}) and obtain $-a_2^+ < 0$, a contradiction.

In conclusion $l=2$, and \eqref{999}, \eqref{1002} imply that $m=1$ and $U=U_{0,k}$ for some $k$.

\smallskip

It remains to show that $U_{0,k}$ is a minimizer. Let $V$ be a perturbation of $U$ which coincides with $U$ outside a compact interval that contains the origin, say $[-1,1]$. Let $\rho$ be such that $u_1 >0 \in (\rho,1]$, $u_1(\rho)=0$ (and therefore $U(\rho)=0$). Then $W(V) \ge 1$ in $[\rho,1]$, hence
$$ J_N(V,[-1,1]) \ge J_N(\bar V, [-1,1]) $$   
where $\bar V$ is the harmonic replacement of $V$ in $[\rho,1]$, with $\bar V$ extended to be $0$ in $[-1,\rho]$. On the other hand, the right hand side is minimized when $\rho=0$ as can be seen from Remark \ref{R1d}.

\end{proof}

If we consider minimality in the more restrictive class of continuous deformations of the graph of $U$ in $\R \times \R$ (which do not change the topology of the junctions), then there are other cones in 1D. For example, for $N=3$,
$$ u_1=  x^+ , \quad u_2=0, \quad u_3=-x^+,$$
is a {\it minimizing} cone with a triple junction at $0$. Notice that the Euler -Lagrange equation \eqref{EL1d} is satisfied. 
 The minimizer in $H^1$ with the same boundary data in $[-1,1]$ of the example above is (see Figure \ref{fig3}) 
 $$ u_1= \lambda \left  (x -1 + \frac 1 \lambda \right )^+, \quad u_2= - \frac 12 u_1 + \mu \left(x-1 + \frac {1}{2 \mu}\right )^+, \quad u_3=-u_1-u_2, $$ 
 $$ \lambda=\frac { \sqrt 2}{\sqrt 3}, \quad \mu = \frac {1}{ \sqrt 2 }.$$

\begin{figure}[h] 
\includegraphics[width=0.35 \textwidth]{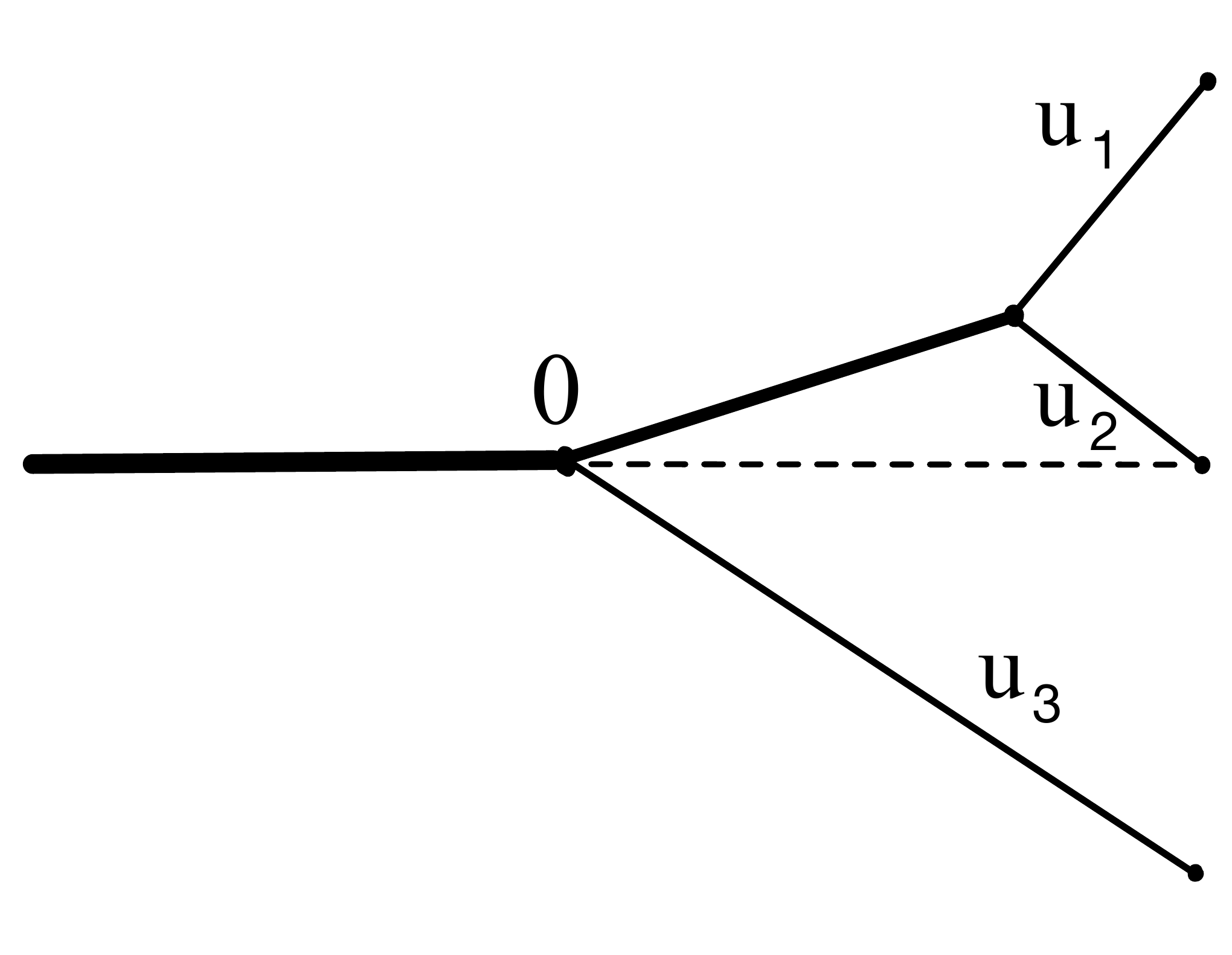}
\caption{A minimizer for $N=3$, $n=1$.} 
   \label{fig3}
\end{figure}


\section{Free boundary regularity near one-dimensional cones}\label{Sec6}

In this section we prove the partial regularity result for the free boundaries $\Gamma_i$, in dimension $n \geq 2.$ For convenience we restate Theorem \ref{prI} from the Introduction.

 \begin{thm}[Partial regularity]\label{pr}
Let $U$ be a minimizer in $B_1$. The free boundaries $\Gamma_i$'s are analytic and disjoint from one another outside a closed set $\Sigma$ of singular points of Hausdorff dimension $n-2$, and
$$\mathcal H^{n-1}(\Gamma_i \cap B_{1/2}) \le C.$$
\end{thm}

In order to obtain Theorem \ref{pr}, we first establish a series of lemmas. First we write a dimension reduction lemma for minimizers depending on fewer variables.

\begin{lem}$U(x)$ is a minimizer for $J_N$ in $B_1 \subset \R^n$ if and only if its trivial extension in one more variable is a minimizer for $J_N$ in $B_1 \times \R \subset \R^{n+1}$.
\end{lem}

The proof of this lemma is standard and we omit the details. 

Next, we obtain a lower bound of the energy of non trivial cones. Recall that cones are 1-homogeneous functions.

\begin{lem}[Least energy cones]\label{L41}
Let $U$ be a nonzero cone in $\R^n$, with $\sum u_i=0$. Then
$$ \Phi(U) \ge \frac 12 |B_1|,$$
with equality if and only if $$U= U_{0,k}(x \cdot \nu),$$ for some $k <N$, and some unit direction $\nu$.
\end{lem}

\begin{proof} We prove this by induction on the dimension $n$. The case $n=1$ was addressed in the previous section.

If $\{|U|=0\}$ does not contain a ray, then $W(U) \ge 1$, and $\Phi(U) \ge |B_1| > \frac 12 |B_1|$. If $|U|=0$ contains a ray, then we can choose $x_0 \in \p\{|U|>0\} \cap \p B_1$, and by the monotonicity formula obtain
$$ \Phi(U)= \lim_{r \to \infty} \Phi_{x_0}(U,r) \ge \Phi_{x_0}(U,0+),$$
with equality only if $U$ is constant in the $x_0$ direction. Here $\Phi_{x_0}(U,0+)$ represents the limit of the $\Phi$ functional centered at $x_0$ and its value coincides with the energy of a blow-up cone at $x_0$. This cone is constant in the $x_0$ direction and by the previous lemma it is the extension of an $n-1$ dimensional cone to $\R^n$. Now the conclusion follows easily from the induction hypothesis. 
\end{proof}

We state the $\eps$-regularity theorem for our problem.

\begin{prop} \label{GK}Assume that $U$ is a minimizer of $J_N$ in $B_1$ with $$0 \in \{|U|>0\}, \quad \quad \sum u_i=0,$$ 
and
$$ \Phi(U,B_1) \le \frac 12 |B_1| + \delta,$$
for some small $\delta$ universal. Then, there exists $k<N$ such that  $\Gamma_k \cap B_{1/2}$ is an analytic hypersurface and $\Gamma_i \cap B_{1/2}=\emptyset$ if $i \ne k$.
\end{prop}

We start by proving the following result.

\begin{lem}\label{L42}
Assume that $U$ satisfies the hypotheses of Proposition $\ref{GK}$. There exists $k$ such that in $B_{3/4}$ 
$u_1=..=u_k$, $u_{k+1}=..=u_N$, i.e. $\Gamma_i \cap B_{3/4}=\emptyset$ for all $i \ne k$.
\end{lem}

\begin{proof} 
First we show that there exists a unit direction $\nu$ such that in $B_{7/8}$,
\begin{equation}\label{4000}
|U-U_{0,k}(x \cdot \nu)| \le \eta(\delta) \quad \quad \mbox{ with $\eta(\delta) \to 0$ as $\delta \to 0$.}
\end{equation}
This follows by compactness from Lemma \ref{L41}. Indeed, if $U_m$ is a sequence of minimizers in $B_1$ which satisfy the hypotheses with $\delta=\delta_m \to 0$, then up to subsequences, $U_m \to \bar U$ uniformly on compact sets of $B_{7/8}$ and
$$0 \in \p\{|\bar U|>0\}, \quad \quad \Phi(\bar U,B_{7/8}) = \lim \Phi(U_m,B_{7/8}) \le \frac 12 |B_1|.$$  
Lemma \ref{L41} implies that $\bar U$ is a cone of least energy, i.e. $\bar U = U_{0,k} (x \cdot \nu)$ for some $k$, and \eqref{4000} is proved.

Conclusion \eqref{4000} holds also for the rescalings $U_r(x) =r^{-1}U(xr)$ with $k$ and $\nu$ depending on $r$. However, the continuity of the $U_r$'s imply that $k$ is in fact independent of $r$, provided that we choose $\eta(\delta)$ sufficiently small.

On the other hand \eqref{4000} and the non-degeneracy of $|U|$, give that $\{|U|=0\}$ and the half-space $x \cdot \nu <0$ are close in the Hausdorff distance sense in $B_{3/4}$. The same argument as above can be applied at any point $x _0 \in \p \{|U|>0\} \cap B_{3/4}$ instead of the origin, and the conclusion easily follows.
 \end{proof}
 
 \begin{proof}[Proof of Proposition $\ref{GK}$]
 
 Let us assume that $\sum u_i =0$. Then, by Lemma \ref{L42}, $\Gamma_i \cap B_{3/4}= \emptyset$ if $i \ne k$, and $u_k$ is the minimizer for the scalar Bernoulli problem
 $$ \int_{B_{3/4}} \left(\frac{nk}{n-k}|\nabla u|^2 + \chi_{\{u>0\}}\right) dx,$$
 with sufficiently flat free boundary in $B_{3/4}$. The result follows from the classical work of Alt and Caffarelli \cite{AC}.
  \end{proof}

\begin{defn} We say that $x_0$ is a regular point of $\Gamma_k$ if, when restricting to the collections of membranes that coincide at $x_0$, the corresponding blow-up cone is one-dimensional. 

\end{defn}

\begin{proof}[Proof of Theorem \ref{pr}]

If at $x_0 \in \Gamma_k$ we have, $$u_l >u_{l+1}=...=u_m>u_{m+1},$$ 
with $l+1\le k <m$, then $$\tilde U:=(u_{l+1},...,u_m),$$ minimizes the energy locally for the energy $J_{m-l}$ involving $m-l$ membranes. After subtracting the average from $\tilde U$, we reduce to the case $\tilde U(x_0)=0$. If the blow-up cone at $x_0$ is one-dimensional, then $x_0$ is a regular point. 

Then, by Proposition \ref{GK}, $\Gamma_k$ is an analytic surface and $\Gamma_i=\emptyset$ near $x_0$, with $i \in \{l+1,...,m \}$. Furthermore, Proposition \ref{GK} implies that if $U_m \to U$ then $\Gamma_{m,k}$ converges to $\Gamma_k$ in $C^{1,\alpha}$ sense near $x_0$. 

Now the theorem follows from the standard dimension reduction argument of Federer, see \cite{F}. We omit the details.

\end{proof}

\section{2 dimensional cones for $N=3$ membranes}\label{Sec7}

In this section we study nontrivial cones $U$ which are not one-dimensional. We restrict to the case $N=3$ membranes in $n=2$ dimensions.

We introduce two cones $V_0$ and $V_s$ which turn out to be the only one-homogenous functions which minimize the energy locally at points outside the origin.

 We use polar coordinates $(r, \theta)$, and view sectors as intervals in the variable $\theta$. We also use the notation
$$ e_\theta=(\cos \theta, \sin \theta).$$

\begin{defn} The cone $V_0=(v_{0,1}, v_{0,2}, v_{0,3})$ is defined as
$$v_{0,1}(x)=
\left\{
\begin{array}{l}
\sqrt{\frac 23} \, \, x \cdot e_{-\frac \pi 6}, \quad \quad \quad \quad \quad \quad \mbox{if} \quad \theta \in [-\frac 23  \pi,  \frac 16 \pi ], \\

\

\\

\sqrt{\frac 1 6} \, \, x \cdot e_{\frac \pi 6}, \quad \quad \quad \quad \quad  \quad \mbox{if} \quad \theta \in [\frac 1 6 \pi,\frac 23 \pi],\\

\

\\

0 \quad \quad \quad \quad \quad \quad \quad \quad \quad \quad \quad  \mbox{otherwise},\\

\end{array}
\right .
$$

$$ v_{0,3}(x)= - v_{0,1}(x_1,-x_2), \quad \quad v_{0,2}=-v_{0,1}-v_{0,2}.$$

The cone $V_s=(v_{s,1},v_{s,2}, v_{s,3})$ is defined as
$$ v_{s,1}(x)= \frac{1}{\sqrt {10} }\, \max \left \{  |x_1|, 2 |x_2| \right \}, $$
$$ v_{s,3}(x)= - v_{s,1}(x_2,x_1), \quad \quad v_{s,2}=-v_{s,1}-v_{s,3}. $$
\end{defn}

\begin{figure}[ht] 
\includegraphics[width=0.45 \textwidth]{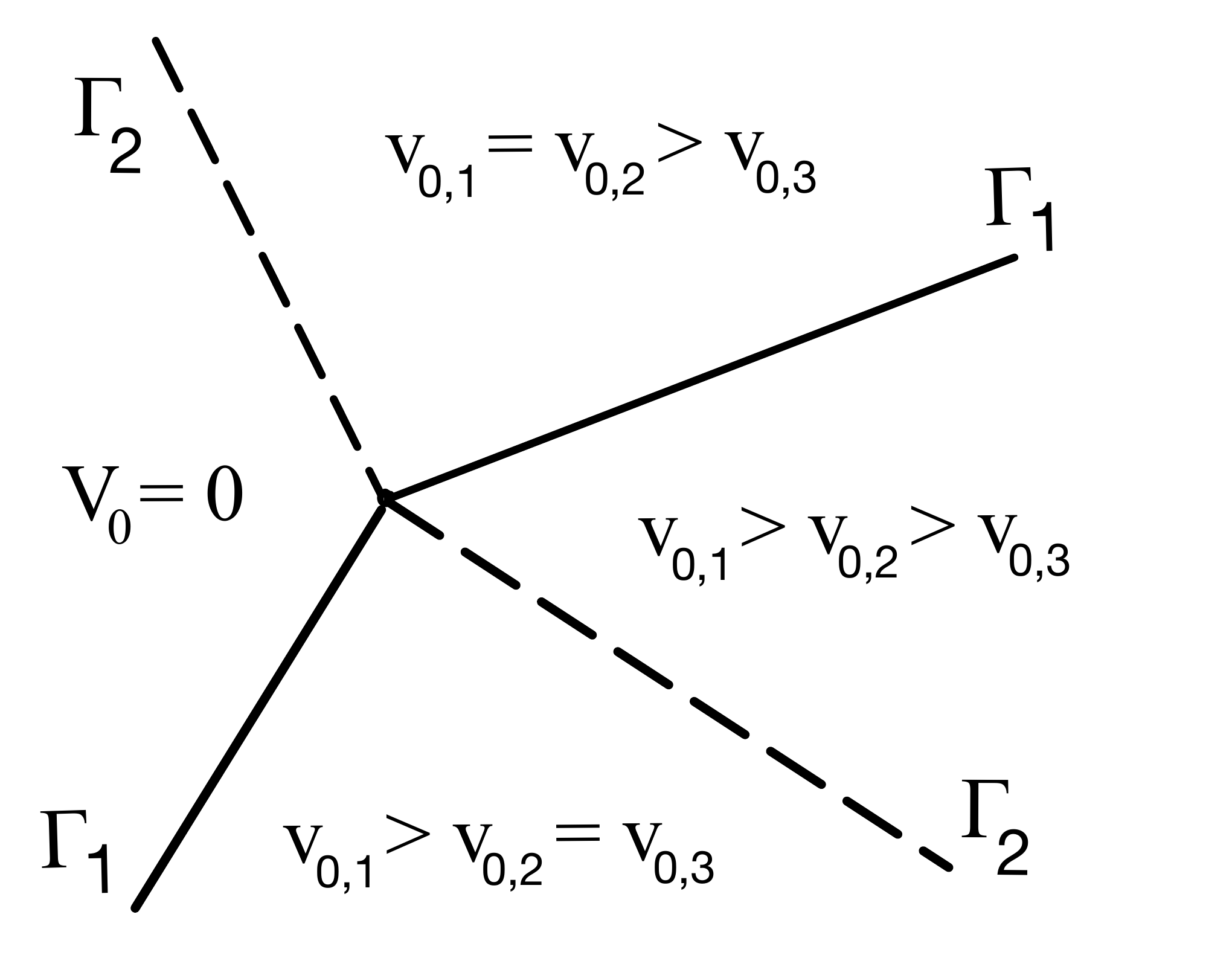}
\hfill
\includegraphics[width=0.45 \textwidth]{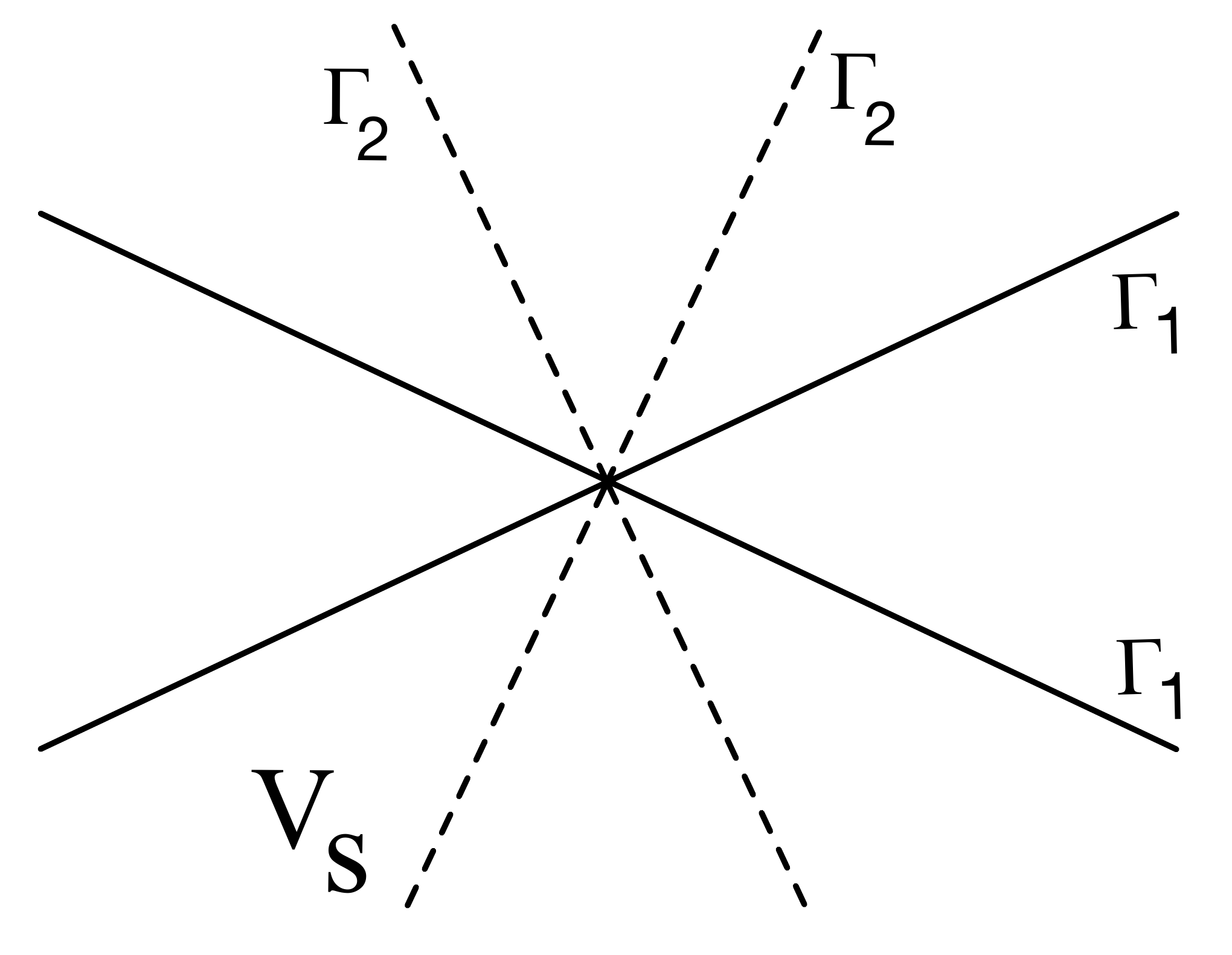}
\caption{ The free boundaries for the cones $V_0$ and $V_s$} 
   \label{fig1}
\end{figure}

\begin{prop} Let $N=3$ and $U$ be a cone with $\sum u_i=0$ which minimizes $J_N$ locally at points on the unit circle $\p B_1$, and assume that $U$ is not one-dimensional. Then, up to rotations $U=V_0$, or $U=V_s$.

\end{prop}

\begin{proof} By the partial regularity Theorem \ref{pr}, the free boundaries $\Gamma_1$ and $\Gamma_2$ of $U$ consist of at most a finite number of distinct rays. Each $u_i$ is linear in the sectors determined by these rays, and on $\p B_1$, the membranes behave as in the 1 dimensional case at the points on $\Gamma_i \cap \p B_1$. More precisely, if at such a point all three membranes coincide, i.e. $U=0$, then on one side they all coincide and on the other they split with multiplicities 1 and 2, and slopes $\sqrt {\frac 23}$ and $\sqrt{ \frac 16} $ respectively. If at such a point only two membranes coincide, then on one side the two membranes coincide, while on the other they split and the slopes have a jump of size $\pm \frac{1}{\sqrt 2}$.

Let $S$ be a sector where $\{u_1>u_2\}$ which, after a rotation, corresponds to an interval $(0, \theta_0)$ in polar coordinates. Then
$$ u_1 = f \cdot x \ge 0 \quad \quad \mbox{in $S$,}  $$
for some non-zero vector $f$. Notice that $\theta_0 \le \pi$. 

We distinguish 2 cases depending on whether or not $u_1$ vanishes on $\p S$.

\

{\it Case 1:} $u_1 >0$ on $\p S \setminus \{0\}$. 

In this case we show that $f$ bisects the sector $S$, and $\theta_0> 2 \pi /3$. Moreover, $f$ bisects also the sector $\{u_2=u_3\} \cap S$.

\

From the 1 dimensional analysis, in the connected sector $\{\theta \in (0, \theta_1) \}\subset S$ where $u_1>u_2>u_3,$ we have
$$ u_2 = (f - \sqrt 2 \, e_{\frac \pi 2}) \cdot x, \quad \quad u_3 = (\sqrt 2 \, e_{\frac \pi 2} - 2 f) \cdot x.$$
Notice that $\theta_1 < \theta_0$, since otherwise $\theta_0=\pi$ and we contradict that $u_1>0$ on $\p S \setminus \{0\}$. 
The two linear functions corresponding to $u_2$ and $u_3$ in $[0, \theta_1]$ intersect on the line
$$ (3f - 2 \sqrt 2 \, e_{\frac \pi 2}) \cdot x =0,$$
hence the ray $\{\theta=\theta_1\} \subset \Gamma_2$ is obtained by intersecting this line with $S$. 

The Euler-Lagrange equation on $\Gamma_2$ implies that
\begin{equation}\label{4004}
 |3f-2 \sqrt 2 \, e_ {\frac \pi 2}|= \sqrt 2,
 \end{equation}
i.e. $ \frac{3}{\sqrt 2} f $ belongs to the unit circle centered at $2 e_{\frac \pi 2}$, which means that the angle between $f$ and $e_ {\frac \pi 2}$ is at most $\pi /6$. 

There is one more unit direction besides $e_{\frac \pi 2}$ which satisfies the equality above, and it must be the inner normal $e_{\theta_0 - \frac \pi 2}$ to $S$ on the other ray of $\p S$.  Clearly the vector $f$ makes the same angle with these 2 inner directions to $S$, hence $f$ bisects $S$. Notice that the sector $S$ is uniquely determined by $f$.

We remark that the angle between $f$ and $e_2$ cannot be $\pi/6$. Otherwise the ray $\theta=\theta_0$ has the direction of $f$ which means that the set $\{u_2=u_3\} \cap S$ consists of a single ray which bisects $S$, and then $U$ cannot be locally minimizing by the 1 dimensional analysis. 

\

{\it Case 2:} $u_1$ vanishes on the ray $\theta=0$.

In this case we show that $U$ coincides with a $\frac \pi 6$ rotation of $V_0$.

\

From the 1 dimensional analysis, 
$$ u_1=\sqrt{\frac 23} \, \,  (x \cdot e_{\frac \pi 2})^+, \quad u_2=u_3 = - \frac 12 u_1,$$
on a sector near the ray $\theta=0$. Let $\{\theta \in (0,\theta_1]\}$ be a connected sector in $S$ where $u_2=u_3$. We claim that 
$$\theta_1 < \theta_0.$$
Otherwise $u_2=u_3=- \frac 12 u_1$ throughout $S$, which means that $S$ is a half-space, i.e. $\theta_0= \pi$. Since $u_1$ is subharmonic, and vanishes near the endpoints of the interval $[\pi,2\pi]$ we conclude that it vanishes in the whole interval, i.e. the formulas above hold in $\R^2$ and $U$ is one-dimensional, a contradiction. 

We remark that $u_1>u_2>u_3$ in the remaining sector $\theta \in (\theta_1,\theta_0)$ of $S$. Then, in this sector we have 
$$ u_1=\sqrt{\frac 23} (x \cdot e_ {\frac \pi 2})^+, \quad u_2= -\frac 12 u_1 + \frac{1}{\sqrt 2} (x \cdot e)^+, \quad  u_3= -\frac 12 u_1 - \frac{1}{\sqrt 2} (x \cdot e)^+, $$
with $e:=e_{\theta_1+ \frac \pi 2}$. 

The Euler-Lagrange equation at $\theta=\theta_0$ gives 
$$ \left|\frac 32 \sqrt{\frac 23} e_{\frac \pi 2} - \frac{1}{\sqrt 2} e \right|=\sqrt 2 $$
which means that $e = e_\pi $, i.e. $\theta_1=\pi/2$. Then the ray $\theta=\theta_0$ is given by the direction of
$$  \frac{\sqrt 3}{2} e+ \frac 12 e_{\frac \pi 2},$$
along which $u_1$ and $u_2$ intersect, i.e. $\theta_0=5 \pi /6$.

The direction $\theta=\theta_0$ is also the one of the gradient of $-u_3$ in the connected sector $S':=(\theta_1,\theta_2)$ where $\{u_2>u_3\}$. 
We claim that $U$ vanishes on $\theta = \theta_2.$ Otherwise we are in case 1) for $-u_3$ in the sector $S'$, and deduce that the angle of $S'$ is $2(\theta_0-\theta_1)=2 \pi /3$ which is a contradiction. 

Thus we are in case 2) for $-u_3$ and $S'$ which means that $\theta_2=\theta_0+ \pi/2$. In the remaining sector $\theta=[\theta_2,2\pi]$, $u_1$ (and therefore $U$) must vanish since it is $0$ on the boundary and it is subharmonic in the interior. In conclusion $U$ is a $\pi /6$ rotation of $V_0$.

\

If $u_1$ and $u_3$ do not vanish on any ray, then we are in Case 1) for each sector of $U$, and then the connected sectors of $\{u_1>u_2\}$ and $\{u_1=u_2\}$ occur in a periodic pattern along the circle. Since $\theta_0 \in (\frac 23 \pi, \pi)$, we can only have four such sectors. Then the angle between the bisectors of two such consecutive sectors, $f$ for $\{u_1>u_2\}$ and $f - \frac{1}{\sqrt 2} e_ {\frac \pi 2}$ for $\{u_1=u_2\}$, must be $- \pi /2$.  This together with \eqref{4004} determines $f$ uniquely as
$$ f= \frac{\sqrt 2}{5} (e_0 + 2 e_{\frac \pi 2}),$$
and then it follows that $U$ is a rotation of $V_s$.
\end{proof}

Next we show that we can lower the energy of $V_s$ by using a compact perturbation near the origin. Our proof does not use the precise form of $V_s$ but only that it has an axis of symmetry.

\begin{prop}
$V_s$ is not a minimizer of $J_N$ in $B_1$.
\end{prop}

\begin{proof}
Assume by contradiction that $V_s$ is a minimizer of $J_N$ in $B_1$. 

We take the domain deformation
$$x \mapsto y:=\Psi(x)=x + \eps  \varphi(|x|) e_1,$$
with $\varphi(|x|)$ a radial function with compact support in $B_1$, and
with $\|\eps D \Psi \|_{L^\infty} <1$. Then, using the Lipschitz continuity of $V_s$, we find as in the proof of Lemma \ref{FV} that
$$ U(y):=V_s(x(y)),$$
satisfies
$$ J_N(U,B_1) = J_N(V_s,B_1) + \int_{B_1} L(\eps D \Psi, V_s) + O(\eps^2 |D \Psi|^2) dx,$$
with $L$ linear in the first argument.
Since $V_s$ is minimizing, it follows that
$$\int_{B_1} L(\eps D \Psi, V_s) dx =0,$$
hence
$$J_N(U,B_1) - J_N(V_s,B_1) \le C \eps^2 \int_{B_1} |\nabla \varphi|^2 dx =o(\eps^2)$$
provided that we choose the logarithmic cutoff 
$$\varphi(r):= \frac{\max\{\log r, \log (2\eps)\}}{\log (2\eps)}.$$
Let $$U_1 := U(-|x_1|,x_2), \quad \quad U_2:=U(|x_1|, x_2)$$ and notice that $U_1$, $U_2$ have the same boundary data as $U$, hence
$$J_N(V_s,B_1) \le J_N(U_1,B_1), \quad J_N(V_s,B_1) \le J_N(U_2,B_1) $$
$$ J_N(U_1,B_1) + J_N(U_2, B_1) = 2 J_N(U,B_1) \le 2 J_N(V_s , B_1) + o(\eps^2),$$
which gives $$J_N(U_1,B_1) \le J_N (V_s,B_1) + o(\eps^2).$$
Notice that in $B_{2 \eps}$, $U$ is a translation of $V_s$ by the vector $\eps e_1$, and $U_1$ and $U_2$ are obtained by reflecting this translation with respect to the line $x_1=0$.

We claim that the minimizer $\bar U_1$ with the boundary data of $U_1$ in $B_\eps$ decreases the energy of $J_N(U_1,B_\eps)$ by $c \eps^2$, $c>0$.  
By the scaling of $\eps^{-1}$ factor, this is equivalent to prove the claim for $\eps=1$. This means that, we first translate $V_s$ by $e_1$ and take its values in $B_1 \cap \{ x_1<0\}$, and then reflect them evenly with respect to $\{x_1=0\}$. We need to show that the resulting function is not minimizing $J_N$ in $B_1$. Indeed, it suffices to look at this function on the $x_1$ axis. On this line we have $u_1=u_2 >0$, $u_3<0$ and at the origin $u_3$ is not smooth, therefore it is not a minimizer.

\end{proof}

Next we prove that $V_0$ is a minimizing cone in $\R^2$. 
\begin{prop}\label{PV0}
$V_0$ is a minimizer of $J_N$.
\end{prop}

First we state a result about the continuity of $U$ near the boundary.

\begin{lem}
Assume that $\Omega$ is a Lipschitz domain and the boundary data of $U$ is Lipschitz. Then $U \in C^\alpha(\overline \Omega)$, for some $\alpha>0$.
\end{lem}

The proof is standard and follows as in the interior case. We omit the details.

\begin{lem}\label{L7.6}
Assume that $0 \in \p \Omega$ and $\Omega$ is a conical domain near the origin. If the boundary data of $U$ is homogenous of degree one near $0$, then the rescalings $U_r(x)=U(rx)/r$ of $U$ converge on subsequences as $ r \to 0$ to a global minimizing cone (defined in the conical domain). 
\end{lem} 

This follows from the Weiss monotonicity formula which still applies in this setting. We only use Lemma \ref{L7.6} when $\Omega = B_1^+ \subset \R^2$.

\begin{lem}\label{LBd0}
Assume that $N=3$, $\Omega=B_1^+=B_1 \cap \{ x_n >0\}$, $\sum u_i =0$, and $U$ vanishes on $x_n=0$. If 
$$ \max\{u_1, |u_3|\} \le \frac{1}{ \sqrt 6} \, x_n,$$
then $U \equiv 0$ in a neighborhood of the origin.
\end{lem}

\begin{proof}

By Lemma \ref{L7.6} and non-degeneracy, it suffices to prove the statement when $U$ is a cone. Since $u_1 \ge 0$ is subharmonic, one-homogenous, and vanishes on $x_n=0$, it must be a linear function in the $x_n$ variable. The same is true for $u_3$ and therefore also for $u_2$. By dimension reduction, it suffices to show the statement in dimension $n=1$. The growth hypothesis implies the slopes of $u_1$, $u_2$ and $u_3$ at the origin are strictly less that $1 / \sqrt 6$, and they need to vanish by Lemma \ref{L1da}.
\end{proof}

\begin{proof}[Proof of Proposition $\ref{PV0}$]

We prove that $V_0$ is a minimizing cone by constructing a minimizer $U: \Omega \to \R$ for which $\Gamma_1 \cap \Gamma_2 \cap \Omega \ne \emptyset$. By the results above, $V_0$ is the only possible blow-up profile at an intersection point, and it is therefore minimizing. 

Let $\Omega:= [-M, M] \times [-1, 1]$, for some $M$ large to be specified later, and let $U$ be a minimizer of $J_N$ with the boundary data given by
$$ \varphi(x_1) U_{0,1}(x_2) + (1-\varphi(x_1)) U_{0,2}(x_2),$$
where $U_{0,1}$, $U_{0,2}$ are the one-dimensional solutions in the case $N=3$, and $\varphi(x_1)$ is a smooth non-decreasing function which is equal to 0 when $x_1\le -1$ and equal to $1$ when $x_1 \ge 1$.
 Thus $U=0$ on $\p \Omega \cap \{x_2\le 0\}$, $U \ne 0$ on the remaining part of the boundary. Moreover, since $u_1$, $-u_3$ are subharmonic we have
 $$ \max\{u_1, - u_3\} \le \frac{1}{\sqrt 6} (1+x_2),$$
  and by Lemma \ref{LBd0} we find that the boundary points
$(-M,M) \times \{-1\} $ are interior to the coincidence set $\{|U|=0\}$.
We have
$$ J_N(U, \Omega) \le(M-1) (J_N(U_{0,1}, [-1,1]) + J_N(U_{0,2}, [-1,1]) + C,$$
and since $U_{0,2}$ is minimizing in one dimension we find $$ J_N(U, \Omega \cap \{x_1 < -1\}) \ge  (M-1)J_N(U_{0,2}, [-1,1]),$$ thus
$$ J_N(U, \Omega \cap \{x_1 > 1\}) \le  (M-1)J_N(U_{0,1}, [-1,1])  + C.$$

On the other hand for each $y \in (1,M)$ for which $$\|U(y, \cdot) - U_{0,1}(\cdot)\|_{L^\infty[-1,1]} \ge \delta,$$ we have
$$ J_N(U, \mathcal D) \ge J_N(U_{0,2}(x_2), \mathcal D) + \sigma(\delta), \quad \mathcal D:=[y-\delta,y+ \delta] \times [-1,1], $$
for some $\sigma(\delta)>0$, hence if $M= M(\delta)$ is large enough we can find a $y \in (M/4,M/2)$ such that 
$$ |U - U_{0,1}(x_2)| \le \delta  \quad \mbox{in} \quad \mathcal R:=[y-1,y+1] \times [-1,1],$$
which, by Proposition \ref{GK}, gives that the coincidence set $\{|U|=0\}$ has $\Gamma_1$ as its boundary in $\mathcal R$, while $\Gamma_2= \emptyset$. The same conclusion holds for some $y \in (-M/2,-M/4)$ with the roles of $\Gamma_1$ and $\Gamma_2$ interchanged.

Now we can conclude that $\{|U|=0\}$ has a boundary point in $\Gamma_1 \cap \Gamma_2 \cap \Omega$. Indeed, otherwise $\p \{ |U|=0\} \cap \Omega$ is a smooth curve which locally is in either $\Gamma_1$ or $\Gamma_2$ and which does not intersect the lines $x_2=1$ or $x_2=-1$ in the region $|x_1| \le M/2$. This is topologically impossible and we reached a contradiction.

Finally we show that $V_0$ is the unique minimizer with its own boundary data. If $U$ is another minimizer with the same energy and boundary data in $B_1$, then $U$ extended by $V_0$ outside $B_1$, is a minimizer in any compact subset of $\R^n$. As above, we find that there exists $x _0 \in \Gamma_1 \cap \Gamma_2$, and since the tangent cone at $x_0$ has the same energy with the cone at infinity we find that $U$ is a cone with vertex at $x_0$, i.e. $U = V_0$.
\end{proof}

\section{Regularity of cones near $V_0$}\label{Sec8}

In this section we assume that we are in $n=2$ dimensions. We establish the uniqueness of the blow-up cone at a point in $\Gamma_1 \cap \Gamma_2$ and obtain the regularity of the two free boundaries near such a point.  
\begin{thm}\label{T2d}
Assume $N=3$ and $U$ minimizes $J_N$ in $B_1$. If $V_0$ is the blow-up profile of $U$ at $0$, then $\Gamma_1$ and $\Gamma_2$ are piecewise $C^{1,\alpha}$ curves.
\end{thm}

Since $V_0$ is the only possible blow-up profile for $\Gamma_1 \cap \Gamma_2$ in two dimensions, we know that $U$ is well-approximated by rotations of $V_0$ at all scales. Precisely we may assume that for each $r \in (0,1]$, there exists a rotation matrix $O_r$ such that
\begin{equation}\label{6009}
\|U - V_0(O_r x)\|_{L^\infty(B_r)} \le \eps r,
\end{equation}
for some $\eps$ sufficiently small. Our goal is to show that the value of $\eps$ improves in a $C^\alpha$ fashion with respect to $r$.

Assume that $O_1=I,$ and
set $$w_1:= \frac 1 \eps \left(u_1- \sqrt {\frac 23}\, \,  x \cdot e_{-\frac \pi 6}\right), \quad \mbox{defined in $\{u_1>u_2\}$,}   $$
$$w_3:= - \frac 1 \eps \left(\sqrt {\frac 23} \, \, x \cdot e_{\frac \pi 6} +u_3 \right), \quad \mbox{defined in $\{u_2>u_3\}$.}   $$
We claim that, as $\eps \to 0$, the graphs of the functions $(w_1,w_3)$ converge uniformly on compact sets of $$B_1\setminus \{0\} \times \R \subset \R^{2+1}$$ to the graphs of a limiting pair $(\bar w_1, \bar w_3)$ with $\bar w_1$ defined in the sector $\mathcal S_1$ given by 
$$\mathcal S_1:=\left\{ (r,\theta) \in (0,1) \times (-\frac {2 \pi}{3}, \frac \pi 6)\right\},$$
\begin{figure}[h] 
\includegraphics[width=0.3 \textwidth]{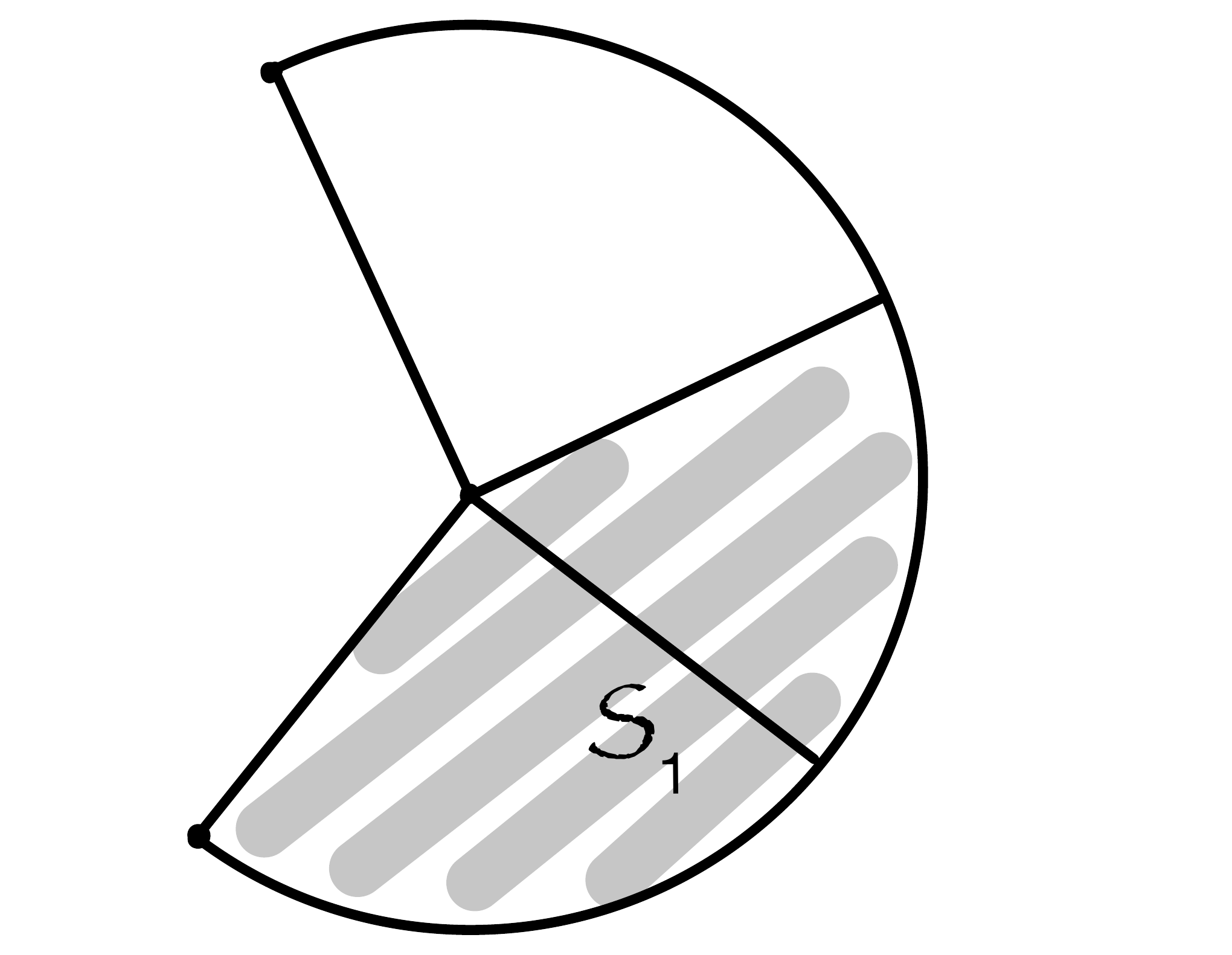}
\caption{The sector $\mathcal S_1$}
   \label{fig5}
\end{figure}
and $\bar w_1$ satisfies
\begin{align}
\label{a1}\triangle \bar w_1 & =0 \quad \quad \quad \mbox{in $\mathcal S_1$,}\\
\label{a2}\p _\nu \bar w_1 & =0 \quad \quad \quad \mbox{on} \quad \p \mathcal S_1 \cap \left\{ \theta=-\frac{2 \pi}{3}\right\},\\
\label{a3}\p _\nu \bar w_1& =\frac 12 \p_\nu \bar w_3 \quad \mbox{on} \quad \p \mathcal S_1 \cap \left\{ \theta=\frac{\pi}{6}\right\},
\end{align}
while $\bar w_3$ is defined in $\mathcal S_3$ with
$$\mathcal S_3:=\left\{ (r,\theta) \in (0,1) \times (- \frac \pi 6,\frac {2 \pi}{3})\right\},$$
and
\begin{align}
\label{a3.5}\triangle \bar w_3&=0 \quad \quad \quad \mbox{in $\mathcal S_3$,}\\
\p _\nu \bar w_3 &=0 \quad \quad \quad \mbox{on} \quad \p \mathcal S_3 \cap \left\{ \theta=\frac{2 \pi}{3}\right\},\\
\label{a4}\p _\nu \bar w_3&=\frac 12 \p_\nu \bar w_1 \quad \mbox{on} \quad \p \mathcal S_3 \cap \left\{\theta=-\frac{\pi}{6}\right\}.
\end{align}

Here the functions $\bar w_1$, $\bar w_3$ are bounded by $-1$ and $1$, and their first derivatives are continuous up to the boundary at the points 
belonging to the lateral sides of the sector. The equations imply that they are in fact smooth up to the boundary at these points.

Moreover, $\bar w_1$, $\bar w_3$ are continuous at $0$,
\begin{equation}\label{6010}
\bar w_1(0)=\bar w_3(0)=0,
\end{equation}
and the convergence of $W=(w_1,w_3)$ to $\bar W=(\bar w_1, \bar w_3)$ is uniform also in a neighborhood of $0$.

Indeed, by the results of Section \ref{Sec6}, $u_1$ solves a Bernoulli one-phase problem near the line $\theta =-\frac 23 \pi$ which is an $\eps$-perturbation of the one-dimensional solution of slope $ \sqrt{2/3}$. Now we can apply the results for the scalar one-phase problem in [D], and find that $w_1$ converges on subsequences as $\eps \to 0$ to a limiting function $\bar w_1$ which satisfies \eqref{a1}-\eqref{a2}. On the other hand 
$$u_1-u_2=2 u_1 + u_3 = \sqrt{\frac 23} \, (2 e_{-\frac \pi 6} - e_{\frac \pi 6}) \cdot x+ \eps(2w_1-w_3)$$ solves a Bernoulli one-phase problem near the line $\theta =\frac \pi 6 $ which is an $\eps$-perturbation of a one-dimensional solution. The same argument as above gives \eqref{a3}. Similarly we obtain the equations \eqref{a3.5}-\eqref{a4} for $\bar w_3$.

Finally, from \eqref{6009} we have $|O_r-O_{r/2}| \le C \eps$, hence $|O_r -I| \le C \eps |\log r|$ which means 
$$\|U - V_0(x)\|_{L^\infty(B_r)} \le C \eps r |\log r|.$$
This implies that $$|w_i| \le C |x| \log |x| \quad \mbox{if} \quad |x| \in [\eps^2 ,1],$$ hence $\bar w_i$ has a $C r |\log r|$ modulus of continuity at $0$ and satisfies \eqref{6010}.

Next we study the linear system \eqref{a1}-\eqref{a4} and notice that it appears as the Euler-Lagrange system for the quadratic functional
\begin{align*}
Q(w_1,w_2): = \frac 32 \int_{\mathcal S_1 \setminus \mathcal S_3} & |\nabla w_1|^2 dx + \frac 32 \int_{\mathcal S_3 \setminus \mathcal S_1}|\nabla  w_3|^2 dx \\
&+ \int_{\mathcal S_1 \cap \mathcal S_3} |\nabla  w_1|^2 + |\nabla w_3|^2 + |\nabla ( w_1- w_3)|^2 dx,   
\end{align*}
acting on pairs $$(w_1,w_3) \in H^1(\mathcal S_1) \times H^1(\mathcal S_3).$$
For simplicity of notation we denote $$ W=(w_1,w_3), \quad \nabla W=(\nabla w_1,\nabla w_3),$$
$$ L^2(\mathcal S):=\{ W| \quad w_i \in L^2(\mathcal S_i) \}, \quad H^1(\mathcal S):=\{ W| \quad w_i \in H^1(\mathcal S_i) \},$$
and define the inner product on $L^2 (\mathcal S)$ as
\begin{align*}
\langle W,V \rangle: = \frac 3 2 \int_{\mathcal S_1 \setminus \mathcal S_3} & w_1 v_1 dx + \frac 3 2 \int_{\mathcal S_3 \setminus \mathcal S_1} w_3 v_3 dx \\
&+  \int_{\mathcal S_1 \cap \mathcal S_3} w_1 v_1 + w_2 v_2 + ( w_1- w_2)(v_1-v_2) dx.   
\end{align*}
The norm induced by the inner product is equivalent to the standard $L^2$ norm on $\mathcal S$.
We also define $\langle \nabla W, \nabla V \rangle$ as above, by replacing the terms $w_i v_j$ by $\nabla w_i \cdot \nabla v_j$.
With this notation
$$ Q(W) = \langle \nabla W, \nabla W \rangle.$$

We consider minimizers of $Q$ which have fixed boundary data on $\p B_1$ or, in other words, we consider {\it harmonic maps} induced by $\langle \cdot, \cdot \rangle$. We establish the $C^{1,\alpha}$ regularity of minimizers of $Q$ near the origin.

\begin{prop}\label{C1a}
Assume that $W \in H^1(\mathcal S)$ is a minimizer of $Q$ among competitors with the same boundary data on $\mathcal S \cap \p B_1$. Then $W \in C^{1,\alpha_0}(\mathcal S)$ and
$$ W(x)= W(0)+ q (e_{\frac \pi 3} \cdot x, - e_{-\frac \pi 3} \cdot x ) + O(|x|^{1+\alpha_0}),$$
for some explicit $\alpha_0 \in (0,1)$.
\end{prop}

The precise formula of $\alpha_0$ is given by
$$ \cos (\frac \pi 3 (1+\alpha_0))= \frac {\sqrt {17}-3}{8},$$
and $\alpha_0 \simeq 0.36$.

First we show that a classical solution to the system \eqref{a1}-\eqref{a4} which remains bounded near the origin minimizes the energy $Q$.

\begin{lem}\label{L5.3}Assume that $\bar W$ is a bounded solution of \eqref{a1}-\eqref{a4}. Then it minimizes the energy $Q$ with respect to perturbations in $H_0^1(\mathcal S)$. 
\end{lem}

\begin{proof}
Indeed, we notice that if $V \in C^1$ vanishes near $\p B_1$ and the origin $0$, then
\begin{align*}
 \langle \nabla \bar W, \nabla V \rangle & =  \langle \triangle \bar W, V \rangle  \\
& + \int_{\mathcal S_3 \cap \p \mathcal S_1 } - \frac 32 v_3 (\bar w_3)_\nu  + v_1 (\bar w_1)_\nu + v_3 (\bar w_3)_\nu +(v_1-v_3) (\bar w_1-\bar w_3)_\nu \, \, d  \sigma \\
& + \int_{\mathcal S_1 \cap \p \mathcal S_3 } - \frac 32 v_1 (\bar w_1)_\nu  + v_3 (\bar w_3)_\nu + v_1 (\bar w_1)_\nu +(v_3-v_1) (\bar w_3- \bar w_1)_\nu \, \, d  \sigma \\
&=0.
\end{align*}

Let $W_0$ be the minimizer of $Q$ in $\mathcal S \cap B_r$ with the same boundary data as $\bar W$ on $\mathcal S \cap \p B_r$ for some $ r \in (0,1)$. We show that $\bar W$ and $W_0$ coincide in $B_r$. 

The computation above gives
$$\langle \nabla (\bar W-W_0), \nabla V \rangle =0,$$
for any $V \in H^1(\mathcal S)$ which vanishes on $\p B_1$ and in a neighborhood of $0$.   

We choose $V= \psi ^2 (\bar W - W_0),$ with $\psi$ a radial cutoff function which is 1 outside $B_\delta$ and vanishes near $0$, and obtain the Caccioppoli inequality
\begin{equation}\label{CCI}
 \langle \psi \nabla (\bar W-W_0), \psi \nabla (\bar W-W_0) \rangle \le C \langle |\nabla \psi|(\bar W-W_0),|\nabla \psi|(\bar W-W_0) \rangle. 
 \end{equation}
Since $\bar W-W_0$ is bounded near the origin and $\|\nabla \psi\|_{L^2}$ can be made arbitrarily small, we obtain 
$$  \langle \nabla (\bar W-W_0), \nabla (\bar W-W_0)\rangle =0,$$
as $\delta \to 0$, which gives the desired conclusion.    
 \end{proof}

Theorem \ref{T2d} follows from the estimate of Proposition \ref{C1a} which applies to $\bar W$. 

\begin{proof}[Proof of Theorem $\ref{T2d}$]
Since $\bar W(0)=0$ then, for all $\rho \le 1/2$,
$$\|\bar W - q (e_{\frac \pi 3} \cdot x, - e_{-\frac \pi 3} \cdot x ) \|_{L^\infty(B_\rho)} \le C \rho ^{1+\alpha_0},$$
with $|q| \le C$, and $C$ universal.

The uniform convergence of the $W$'s to $\bar W$ and the inequality above imply 
$$\|U - V_0 (O x)\|_{L^\infty(B_\rho)} \le C \rho^{1+\alpha_0} \eps \le \eps \rho^{\alpha},$$
with $O$ the rotation of angle $\eps \sqrt{ \frac 32} \, q$, provided that $\alpha < \alpha_0$ and $\rho$ is chosen sufficiently small depending on $\alpha$.
This means that $U$ is approximated in a $C^{1,\alpha}$ fashion by rotations of $V_0$, and the desired conclusion follows by iterating this result.
\end{proof}
 
 It remains to prove Proposition \ref{C1a}.
 
 \begin{proof}[Proof of Proposition $\ref{C1a}$]
 
 We solve the Dirichlet problem for minimizers of $Q$ by the method of Fourier series. 

We investigate the eigenvalues and eigenfunctions of the corresponding $Q$ operator on the unit circle $\p B_1$. Precisely, let 
 $$\mathcal S_i':=\mathcal S_i \cap \p B_1,$$
 and notice that $$\mathcal S_1'=(-b, a), \quad \mathcal S_3'=(-a, b),$$
 where $a<b$ denote $\frac {\pi}{6}$ and $\frac{2\pi}{3}$.

 We define the corresponding spaces $L^2(\mathcal S')$, $H^1(\mathcal S')$, and the inner product
 $ \langle W,V \rangle $ as above. The eigenfunctions $\Phi_k$ and eigenvalues $\lambda_k$ are defined inductively through the 
 Rayleigh quotient formula
 $$ \lambda_{k}:=\min_{W \in (span\{\Phi_1,..,\Phi_{k-1}\})^\perp} \frac{\langle \nabla W,\nabla W \rangle}{\langle W,W \rangle},$$
 and $\Phi_k=W$ is the element of $H^1(\mathcal S')$ which has unit norm in $L^2(\mathcal S')$, where the minimum is realized.  
 Then $\{\Phi_k\}_{k=1}^\infty$ is an orthonormal basis of $L^2(\mathcal S')$. 
 
 The minimizer of the $Q$ functional with boundary data 
 $\Phi_k(\theta)$ on 
 $\p B_1$ is
 $$ r^{ \sqrt { \lambda_k }} \Phi_k(\theta).$$

 We write the Euler-Lagrange equations for an eigenfunction $\Phi_k=W$ of eigenvalue $\lambda_k=\lambda$ and obtain
  \begin{equation}\label{6000}
  w_1 '' + \lambda w_1 = 0  \quad \mbox{in} \quad (-b,-a) \cup(-a,a),
  \end{equation}
  $$ w_1'(-b)_+=0, $$
  \begin{equation}\label{6001}
  \frac 32 w_1'(-a)_- + w_3'(-a)_+ - 2w_1'(-a)_+ =0, \quad w_1'(-a)_+ - 2 w_3'(-a)_+   =0, 
  \end{equation}
together with a similar equations for $w_3$. Here $w'_1(-a)_{\pm}$ denote the left and right limits of the derivative of $w_1'$ at $-a$. 
From \eqref{6001} we find $$w_1'(-a)_-=w_1'(-a)_+=2w_3'(-a)_+,$$ which means that equation \eqref{6000} holds in the full interval $(-b,a)$.
In conclusion $$w_1(\theta) = \mu_1 \cos (\sqrt \lambda (\theta + b)), \quad w_3(\theta)= \mu_3 \cos (\sqrt \lambda (\theta-b) ),$$
and $$ w_1'(-a)= 2 w_3'(-a), \quad w_3'(a)= 2 w_1'(a).$$
This implies that
$$ \sin (\sqrt \lambda (b-a)) = \pm 2 \sin(\sqrt \lambda(b+a)) $$
or
\begin{equation}\label{6011}
 |\sin (3 t)| = 2 |\sin (5t)|, \quad t:= \frac{\sqrt{\lambda}}{6} \pi.
 \end{equation}
This equation has periodic solutions in $t$ of period $\pi$. 
In the interval $[0,\pi)$ we have 9 solutions, two in each intervals of length $\pi /5$. The first one is $t_0=0$, then $$t_1=\frac{\pi}{6}, \quad t_2 \in (\frac \pi 5,\frac \pi 3), \quad t_3 \in (\frac \pi 3,\frac {2 \pi}{ 5}) \quad \quad \mbox{ etc.} $$
The solutions can be computed explicitly since after dividing by $\sin t$ in \eqref{6011}, we end up with two quadratic equations in $4 \cos (2t)$
$$ \beta^2 + 2 \beta - 4 = \pm (\beta +2), \quad \quad \beta:= 4 \cos(2t).$$ 
The corresponding eigenvalues and eigenfunctions are
\begin{align*}
\lambda_1 &=0, \quad \quad \quad W= (1,0), \\
\lambda_2 &=0, \quad \quad \quad W=(0,1), \\
\lambda_3 & =1, \quad \quad \quad W=\left(\cos (\theta + b), - \cos (\theta - b)\right),\\
\sqrt{\lambda_4} & \in (\frac 65,2), \quad  W=\left( \cos\left( \sqrt{\lambda_4} (\theta + b)\right),  \cos \left(\sqrt {\lambda_4}(\theta - b) \right)\right),
\end{align*}
and $\sqrt{\lambda_3} \in (2, \frac{12}{5})$ etc. The eigenvalues corresponding to $t / \pi \in \mathbb Z$ have multiplicity 2, while the others are simple, and notice that $\sqrt {\lambda_k} \sim \frac 35 k $ for $k$ large.

If $\Phi(\theta)$ is the boundary data of $W$, then we decompose it in $L^2(\mathcal S')$ as 
$$ \Phi = \sum \sigma_k \Phi_k$$
with $\sum \sigma_k^2 < \infty$, and write the solution $W$ in $\mathcal S$ in polar coordinates as the series
$$W= \sum \sigma_k r^{\sqrt {\lambda_k}} \Phi_k(\theta),$$
which converges uniformly in compact intervals of $ r \in [0,1)$.

 This gives the desired conclusion with
$$1+\alpha=\sqrt {\lambda_4},$$
and $$ 4 \cos (\frac \pi 3 \sqrt {\lambda_4}) = \frac{\sqrt {17} - 3}{2}.$$
We remark that $\Phi$ is the trace of a function in $H^1(\mathcal S)$ is equivalent to 
$$Q(W)= \langle \nabla W,\nabla W \rangle = \sum \sqrt{\lambda_k} \, \, \sigma_k^2 < \infty.$$
\end{proof}

\section{Regularity for $\Gamma_1 \cap \Gamma_2$ in higher dimensions}\label{Sec9}

In this section we prove a version of Theorem \ref{T2d} in arbitrary dimensions. First we recall the following definitions. 

\begin{defn}
Let $U$ be a minimizer of $J_N$ with $N=3$. We say that $$x_0 \in Reg(\Gamma_1 \cap \Gamma_2)$$ if there exists a blow-up profile at $x_0$ which is a rotation of a two-dimensional cone $V_0$ extended trivially in the remaining variables. 

We say that $$x_0 \in Reg (\Gamma_k)$$ if there exists a blow-up cone $x_0\in \Gamma_k$ which is one-dimensional. \end{defn}

For convenience we restate Theorems \ref{TndI} and \ref{TPRI} from the Introduction which will be proved in this section.
\begin{thm}\label{Tnd}
$Reg(\Gamma_1 \cap \Gamma_2)$ is locally a $C^{1,\alpha}$-smooth manifold of codimension two. Near such an intersection point, each of the free boundaries $\Gamma_1$ and $\Gamma_2$ consists of two piecewise $C^{1,\alpha}$ hypersurfaces which intersect on $Reg(\Gamma_1 \cap \Gamma_2)$.
\end{thm}

As a consequence we obtain the partial regularity result.

\begin{thm}[Partial regularity]\label{TPR} Let $N=3$ and $U$ be a minimizer of $J_N$ in $B_1$. Then
$$ \partial \{|U|>0\}=Reg(\Gamma_1) \cup Reg(\Gamma_2)\cup Reg (\Gamma_1 \cap \Gamma_2) \cup \Sigma',$$
with $\Sigma'$ a closed singular set of Hausdorff dimension $n-3$ and
$$ \mathcal H^{n-1} (Reg(\Gamma_i)\cap B_{1/2}) \le C, \quad \quad  \mathcal H^{n-2}(Reg(\Gamma_1 \cap \Gamma_2)\cap B_{1/2}) \le C.$$

\end{thm}

Theorems \ref{Tnd} and \ref{TPR} are deduced from the next proposition which will be proved later in the section. It states that a minimizer $U$ which is approximated in each ball $B_r$ with $r \in (0,1]$, by a rotation of $V_0$, must be a $C^{1,\alpha}$ deformation of $V_0$. 
\begin{prop}\label{C1alphand}
Let $U$ be a minimizer in $B_1$. Assume that for each $r \in (0,1]$ there exists a pair of orthonormal vectors $\nu^1_r,\nu^2_r$ such that
$$ \|U-V_0(\nu_r^1 \cdot x, \nu_r^2 \cdot x) \|_{L^\infty (B_r)} \le \eps r,$$
for some $\eps \le \eps_0(n)$ small universal. Then there exists $\nu_0^1, \nu_0^2$ such that  
$$ \|U-V_0(\nu_0^1 \cdot x, \nu_0^2 \cdot x) \|_{L^\infty (B_r)} \le C \eps r^{1+\alpha}.$$
\end{prop} 

A direct consequence of this result is that there are no other cones that are sufficiently close to the set of rotations of $V_0$.

\begin{cor}\label{Cco}
Assume that $U$ is a minimal cone and $\|U-V_0\|_{L^\infty (B_1)} \le \eps_0$. Then $U$ is a rotation of $V_0$.
\end{cor}

We show that the hypotheses of Proposition \ref{C1alphand} can be relaxed to require that $U$ is approximated by $V_0$ only in $B_1$ and the the energy of the blow-up cones at the origin does not go below the energy of $V_0$. 

\begin{lem} \label{L6p5} Let $U$ be a minimizer in $B_1$ with $0 \in \p \{|U|>0\}$. Assume that $$\|U-V_0\|_{L^\infty (B_1)} \le \eps, \quad \mbox{and} \quad \Phi_U(0+) \ge \Phi(V_0).$$ Then the only possible blow-up cones for $U$ at $0$ are rotations of $V_0$.
\end{lem}

\begin{proof}
First we notice that the hypotheses imply that $$\Phi_U(\frac 12) \le \Phi(V_0) + C \eps,$$ and then
\begin{equation}\label{6020}
0 \le \Phi_U(r)-\Phi(V_0) \le C \eps, \quad \forall r \in (0, \frac 12).
\end{equation}
Assume by contradiction that the conclusion does not hold for a sequence $U_n$, and $\eps=\eps_n \to 0$. Then, by Proposition \ref{C1alphand}, we can find appropriate dilations $\tilde U_n:= r_n^{-1} U_n(r_n x)$ such that 
$$ dist (\tilde U_n, \mathcal V_0) = \eps_0,$$
where $\mathcal V_0$ represents the collections of cones obtained by rotations of $V_0$ and the distance between $\tilde U_n$ and the elements of $\mathcal V_0$ is measured in $L^\infty(B_1)$. As $n \to \infty$ we may extract a convergent subsequence in $L^\infty(B_1)$ of the $\tilde U_n$'s to a limiting function $\bar U$. Then $\bar U$ must be a cone, since its $\Phi_{\bar U}$ energy is a constant in the radial variable by \eqref{6020}. The distance from $\bar U$ to $\mathcal V_0$ is $\eps_0$, and we contradict Corollary \ref{Cco}.

\end{proof}

Next we use the dimension reduction argument to show that the set of free boundary points $\p \{|U|>0\}$ whose tangent cones have energy strictly between the one dimensional solutions $U_{0,k}$ and the two-dimensional solution $V_0$, has Hausdorff dimension $n-3$.
\begin{lem}\label{L6p6}
The set $$ \mathcal A :=\{x \in \p \{|U|>0\}\cap B_1| \quad \Phi_{U,x}(0+) \in (\Phi(U_{0,k}), \Phi(V_0)) \},$$
has Hausdorff dimension $n-3$.
\end{lem}

Here $\Phi_{U,x}(r)$ denotes the Weiss energy of $U$ in a ball of radius $r$ centered at $x$, and $\Phi_{U,x}(0+)$ its limit as $r \to 0$. The continuity of $\Phi_{U,x}(r)$ with respect to $x$, and $r$ fixed, shows that
$$ \Phi_{U,x}(r) < \Phi(V_0),$$
for all $x \in \p \{ |U|>0\}$ near a point in $\mathcal A$ and $r$ sufficiently small. Thus, it suffices to prove the following lemma.
\begin{lem}
Fix $\delta>0$, and assume that $U$ is a minimizer in $B_3$ and
$$ \Phi_{U,x}(r) \le \Phi(V_0) - \delta, \quad \quad \forall x \in \p \{|U|>0\} \cap \overline B_{1}, \quad \forall r \le 1.$$
Then $$\mathcal H^{n-3+\delta} (\mathcal A \cap \overline B_1) =0.$$ 
\end{lem}
We remark that $\mathcal A \cap \overline B_1$ is a closed set by the regularity result in Proposition \ref{GK}. 

\begin{proof} The proof follows the standard dimension reduction argument of Federer \cite{F}.
Notice that in dimension $n=3$ the set $\mathcal A \cap \overline B_{1}$ is discrete by Proposition \ref{GK}, and the statement is obvious.

We prove the statement by induction on the dimension $n \ge 3$, in two steps. We only sketch the main ideas and leave the details to the interested reader.

{\it Step 1:} Assume the result holds in dimension $n$, then it holds for any $n+1$-dimensional cone $U$ with $\Phi(U) \le \Phi(V_0) - \delta$.

{\it Step 2:} Assume the result holds for cones in dimension $n$, then it holds for a minimizer in dimension $n$.

For Step 1, assume $U$ is a cone in $\R^{n+1}$. We take its restriction to a ball $B_r(x_0)$ with $x_0 \in \mathcal A \cap \p B_1$ and then normalize it to the unit ball after a translation and dilation. The resulting function is uniformly-well approximated as $r \to 0$ by a minimizer that is constant in the $x_0$-direction, for which the induction hypothesis holds. By compactness, this means that, there exists $r_0(n)>0$ small such that if $r \le r_0$ 
then $\mathcal A \cap \p B_1 \cap B_r(x_0)$ can be covered by a finite collection of balls of radii $r_i$ and centers on $\mathcal A \cap \p B_1$ with
\begin{equation}\label{n-3d}
\sum r_i^{n-3 + \delta} \le \frac 12 r^{n-3+\delta}.
\end{equation}
 Step 1 follows by iterating this result a number of times. 
 
 Step 2 is a consequence of the fact that around each point in $\mathcal A$, $U$ is well-approximated by cones at all small scales, and the conclusion holds for these cones by Step 1.
Precisely, by compactness, for each $x_0 \in \mathcal A \cap \overline B_1$ there exists $\delta(x_0)>0$ such that if $r \le \delta(x_0)$ the set $\mathcal A \cap B_r (x_0)$ can be covered by a finite collection of balls of radii $r_i$ that satisfy inequality \eqref{n-3d}. 

Let $\mathcal A_k$ denote the set of points $x_0$ in $\mathcal A$ with the property that $\delta(x_0) \ge \frac 1 k$, and notice that $\mathcal A \subset \cup \mathcal A_k$. On the other hand $\mathcal H^{n-3+\delta} (\mathcal A_k \cap \overline B_1) =0$ since, as in Step 1, we can iterate \eqref{n-3d} for $\mathcal A_k$. Thus, the desired conclusion holds for $\mathcal A$ as well. 
\end{proof}

In view of Lemma \ref{L6p5} and Lemma \ref{L6p6} we obtain a stronger version of Proposition \ref{C1alphand} in which the only hypothesis is that $U$ is approximated by $V_0$ in $B_1$. 
\begin{prop}\label{PLa}
Assume that $U$ is a minimizer in $B_1$, $0 \in \p \{|U|>0\}$, and $$\|U-V_0(x_1,x_2)\|_{L^\infty(B_1)} \le \eps_0,$$
for some $\eps_0$ small universal. Then there exist $\nu_0^1, \nu_0^2$ such that  
$$ \|U-V_0(\nu_0^1 \cdot x, \nu_0^2 \cdot x) \|_{L^\infty (B_r)} \le r^{1+\alpha}.$$
In particular, $\Gamma_1 \cap \Gamma_2 \cap B_{1/2}$ consists only of regular intersection points.
\end{prop}

\begin{proof}
For each $$x \in \mathcal D:=\{x_1=x_2=0\} \cap B_{1/2},$$ we look at the two dimensional plane generated by the first two coordinates of $x$. By topological considerations in this plane, the set $\Gamma_1 \cap \Gamma_2 $ contains at least one point $\bar x$ in the disk of radius $C \eps_0$ around the origin. 

Indeed, our hypothesis and Proposition \ref{GK} imply that in the two-dimensional annulus $ C \eps_0 \le r \le 2 C \eps_0$ the open sets $\{u_1 >u_2\}$ and $\{u_2 > u_3 \}$ are $C^1$ perturbations of the sectors $\mathcal S_1$ and $\mathcal S_3$ defined in Section \ref{Sec8}. By two-dimensional topology, the boundaries of these two open sets must intersect in the disk of radius $C \eps_0$.

Since $dim(\mathcal A) \le n-3$, we find that
$$ \bar x \notin \mathcal A \quad \mbox{ for $\mathcal H^{n-2}$ a.e. $x \in \mathcal D$,} $$
 hence $\Phi_{U,\bar x}(0+) \ge \Phi(V_0)$ and we can apply Proposition \ref{C1alphand} at $\bar x$. 

The conclusion follows since the set of such $x$'s is dense in $\mathcal D$.

\end{proof}

Theorem \ref{Tnd} follows easily from Proposition \ref{PLa} and we omit the details.

Regarding Theorem \ref{TPR}, we notice that
$$\Sigma':= \p \{|U| >0\} \setminus \left(Reg(\Gamma_1) \cup Reg(\Gamma_2) \cup Reg(\Gamma_1 \cap \Gamma_2) \right) ,$$
is a closed set according to Proposition \ref{PLa} and Theorem \ref{pr}. The dimension reduction argument as in Lemma \ref{L6p6} implies that $dim (\Sigma') \le n-3$, and rest of Theorem \ref{TPR} follows by standard techniques.

It remains to prove Proposition \ref{C1alphand}. The considerations at the beginning the previous section remain valid, and they reduce the proof of Proposition \ref{C1alphand} to the validity of $C^{1,\alpha}$ estimates for bounded solutions of the elliptic system \eqref{a1}-\eqref{a4}.

The rest of the section is devoted to establish Proposition \ref{C1a} in arbitrary dimensions, see Proposition \ref{C1an} below.

We introduce some notation. We denote by
$$x=(x',x''), \quad x'=(x_1,x_2), \quad x''=(x_3,..,x_n),$$
$$(r,\theta) \quad \mbox{the polar coordinates for $x'$},$$
$$ \mathcal S_1 := \{\theta \in (-4a, a)\} \cap B_1, \quad \mathcal S_3 := \{\theta \in (-a, 4a)\} \cap B_1, \quad a:= \frac \pi 6,$$
and recall the definitions of $L^2 (\mathcal S)$, $H^1(\mathcal S)$, $\langle W,V \rangle$, $\langle \nabla W, \nabla V \rangle$, and $Q$ from the previous section.
We establish the $C^{1,\alpha}$ regularity of minimizers of $Q$.

\begin{prop}\label{C1an}
Assume that $W \in H^1(\mathcal S)$ is a minimizer of $Q$ among functions with the same trace on $\p B_1$. 
Then $W \in C^{1,\alpha}(\mathcal S)$ and
$$ W(x)= W(0)+  \left((qe_{\frac \pi 3},\nu_1'') \cdot x, (- qe_{-\frac \pi 3},\nu_3'') \cdot x \right) + O(|x|^{1+\alpha}),$$
for some $\alpha \in (0,1)$.
\end{prop}

As in the previous section, Proposition \ref{C1alphand} follows from Proposition \ref{C1an} provided we show that bounded solutions of the system \eqref{a1}-\eqref{a4} do minimize the energy, as in Lemma \ref{L5.3} for $n=2$.

\begin{lem}\label{L5.3n}Assume that $\bar W \in L^\infty$ solves the system \eqref{a1}-\eqref{a4} in the classical sense in the domain
$$\overline{\mathcal S} \setminus (\{x'=0\} \cup \p B_1).$$ Then it minimizes the energy $Q$ with respect to perturbations in $H^1(\mathcal S)$ which vanish near $\p B_1$. 
\end{lem}

\begin{proof} The proof is essentially the same as the one of Lemma \ref{L5.3}. However, in order to justify the existence of a minimizer $W_0$ with the same boundary data as $\bar W$ on $\p B_r$ we need to show first that $\bar W \in H^1(\mathcal S \cap B_r)$. 

Notice that for any $V \in C^1$ that vanishes near $\{x'=0\} \cup \p B_1$, we have
$$ \langle \nabla \bar W, \nabla V \rangle =0.$$
Then the Caccioppoli inequality
$$\langle \varphi \nabla \bar W, \varphi \nabla \bar W \rangle \le C \langle |\nabla \varphi| \bar W, |\nabla \varphi| \bar W \rangle,$$
holds if $\varphi$ vanishes near $\{x'=0\} \cup \p B_1$. We choose
$$\varphi(x)= \psi(x') \eta(x),$$
with $\psi$ a radial cutoff function which vanishes near the origin and $\psi=1$ when $|x'| \ge \delta$, and $\eta$ a cutoff function which vanishes near $\p B_1$. Since $\bar W \in L^\infty$, it follows that we can take $\psi \equiv 1$ in the limit as $\delta \to 0$, i.e. 
$$\bar W \in H^1 (\mathcal S \cap B_r) \quad \mbox{ for any $r<1$.}$$
We define $W_0$ to be the minimizer of $Q$ in $\mathcal S \cap B_r$ with the same boundary data as $\bar W$ on $\p B_r$. Now the proof of Lemma \ref{L5.3} applies, by taking $\psi=\psi(x')$ depending only on the variable $x'$.
\end{proof}

There are several ways to prove Proposition \ref{C1an}. Here we take advantage of the product structure of the problem and reduce it back to the two-dimensional case. The system is invariant under translations in the $x''$ variable, and then we can estimate higher order derivatives $D_{x''}^\beta W$ through successive iterations. Then 
$$ \triangle W =0 \quad \Longrightarrow \quad \triangle_{x'} W = - \triangle_{x''} W,$$
and the right hand side is well behaved. 

We start with some preliminary estimates.

\begin{lem}\label{Ave} Let $W$ be a minimizer of $Q$ in $H^1(\mathcal S)$. Then,

a) (Caccioppoli inequality) If $\varphi \in C_0^1(B_1)$ then
$$\langle \varphi \nabla W, \varphi \nabla W \rangle \le C \langle |\nabla \varphi| W, |\nabla \varphi| W \rangle,$$

b) $W$ is smooth up to the boundary of $S$ away from $\p B_1\cup \{x'=0\}$, and the Euler-Lagrange equations \eqref{a1}-\eqref{a4} are satisfied in the classical sense,

c) $W \in L^\infty(\mathcal S \cap B_{1/2})$.

\end{lem}

\begin{proof}
Part a) is standard and we skip the details. 

For part b) we remark that in a ball $B_\delta(x_0)$ near a point $x_0 \in \mathcal S_1 \cap \p \mathcal S_3$ the energy can be written as
$$ \frac 3 2 \int_{B_\delta(x_0)} |\nabla w_1|^2 dx + \frac 1 2 \int_{B_\delta(x_0) \cap S_3}| \nabla (2 w_2- w_1)|^2 dx.$$
This shows that $w_1$ is harmonic near $x_0$, and $2 w_2 - w_1$ can be extended harmonically in the whole $B_\delta(x_0)$ by the even reflection across $\p \mathcal S_3$. Hence $W$ and its derivatives can be bounded in $\mathcal S \cap B_{1/2}$  away from the codimension two edge $\{x'=0\}$ in terms of the $L^2(\mathcal S)$ norm of $W$.

For part c) we first show that 
\begin{equation}\label{6012}
\fint _{ \mathcal S \cap \p B_r} |W|^2 dx,
\end{equation}
remains bounded for all $r$ small. For this we prove a mean value inequality with respect to the $L^2(\mathcal S)$ norm: 
\begin{equation}\label{6013}
\langle W(rx), W(rx) \rangle_{\mathcal S \cap \p B_1},
\end{equation}
is monotone increasing in $r$, 
where $\langle \cdot ,\cdot \rangle_{S \cap \p B_r}$ denotes the inner product induced on the sphere $\p B_r$.

Let $\varphi \in C_0^1(B_1)$, $\varphi \ge 0$, and then
$$ \langle \nabla W, \nabla (\varphi W) \rangle =0,$$or
$$ 0 \le  \langle \nabla W, \varphi \nabla W \rangle \le \langle - \nabla \varphi \cdot \nabla W, W \rangle.$$
We take $\varphi \to \chi_{B_r}$ and obtain
$$ 0 \le   \langle W_\nu , W \rangle_{\mathcal S \cap \p B_r}.$$
This means that the derivative in $r$ of the expression in \eqref{6013} is nonnegative, and the claim \eqref{6012} is proved.

From \eqref{6012} and part b) applied in $\mathcal S \cap B_r$, we deduce that $|W|$ is bounded in $B_{1/2} \cap \mathcal S \cap \{x''=0 \}$ by a multiple of its $L^2(\mathcal S)$ norm. The conclusion follows by translating the origin at the other points in $\{x'=0\} \cap B_{1/2}$.  
\end{proof}

\begin{lem}\label{L6.5} Assume that $W \in H^1(\mathcal S)$ is a minimizer of $Q$. Then 
$$\|D_{x''}^\beta W\|_{L^\infty(\mathcal S \cap B_{1/2})} \le C(\beta) \|W\|_{L^2(\mathcal S)},$$
and, for each fixed $x''$, the function $W(\cdot, x'')$ minimizes the two-dimensional energy  
$$\langle \nabla W , \nabla W\rangle_{x'} + \int_{\mathcal S_1} f_1 w_1 dx' + \int_{\mathcal S_3} f_3 w_3 dx',$$ 
for some bounded functions $f_1$, $f_3$.
\end{lem}

\begin{proof}  The discrete differences of $W$ in the $x''$-directions are minimizers of $Q$. By iterating Caccioppoli inequality we obtain that $D_{x''}^\beta W$ minimizes $Q$ and
$$ \|D_{x''}^\beta W\|_{H^1(\mathcal S \cap B_{1/2})} \le C(\beta) \|W\|_{L^2(\mathcal S)}.$$ 
The $L^\infty$ bound follows from Lemma \ref{Ave}, part c). This means that $W(x',0)$ satisfies in the two-dimensions 
$$\triangle _{x'} W \in L^\infty,$$
and the boundary conditions in \eqref{a1}-\eqref{a4} hold in the classical sense on $\p \mathcal S \setminus \{0\}$.
The conclusion follows as in Lemma \ref{L5.3} since $W \in L^\infty$.
\end{proof}

The results of the previous section imply the $C^{1,\alpha}$ estimates in this inhomogeneous setting.

\begin{lem}\label{C1a2}
Assume that $W \in H^1(\mathcal S)$ is a minimizer of the 2 dimensional energy
$$\langle \nabla W , \nabla W\rangle + \int_{\mathcal S_1} f_1 w_1 dx + \int_{\mathcal S_3} f_3 w_3 dx,$$ 
for some bounded functions $f_i$.
Then $W \in C^{1,\alpha}(\mathcal S)$ and
$$ |W(x)- W(0)+ q (e_{\frac \pi 3} \cdot x, - e_{-\frac \pi 3} \cdot x ) | \le C(1+ \|f_1\|_{L^\infty} + \|f_3\|_{L^\infty}) \|W\|_{L^2} \, \, |x|^{1+\alpha},$$ 
in $ \mathcal S \cap B_{1/2},$ for some universal constant $C$. 
\end{lem}

Now we can prove the $C^{1,\alpha}$ estimate in arbitrary dimensions.  

\begin{proof}[Proof of Proposition $\ref{C1an}$]
By Lemma \ref{L6.5} and Lemma \ref{C1a2} (applied to $W$ and $D_{x''} W$) we find that
$$W(x',0)=W(0) +  q (e_{\frac \pi 3} \cdot x', - e_{-\frac \pi 3} \cdot x ') + O(|x'|^{1+\alpha}),$$
$$ D_{x''}W(x',0)= (\nu_1'',\nu_3'') + O(|x'|), \quad \quad |D^2_{x''} W| \le C,$$
which gives the desired conclusion.
\end{proof}

Finally we prove the Schauder estimates in the two-dimensional inhomogeneous setting.

\begin{proof}[Proof of Lemma \ref{C1a2}]

The proof is standard and uses Campanato iterations. We sketch some of the details.

Assume that $\|f_i\|_{L^\infty} \le \delta$, $\|W\|_{H^1} \le \delta$. It suffices to show inductively in $k$ that for each $r=\rho^k$ there exist {\it linear} functions at $0$ of the type
$$ L_r=\left (d_1 + q e_{\frac \pi 3} \cdot x, d_3 - q e_{-\frac \pi 3} \cdot x\right),$$
depending on $r$ which approximate $W$ in $B_r$ such that
$$\left(\fint_{\mathcal S \cap B_r} |W-L_r|^2 dx \right)^{1/2}\le r^{1+\alpha}.$$ 
Indeed, the rescaled function
$$\tilde W(x) = \frac {1}{ r^{1+\alpha}} (W-L_r)(rx),$$
satisfies
$$\fint_{S \cap B_1} |\tilde W|^2 dx \le 1,$$
and $\tilde W$ is a minimizer for a functional as above with $$\|\tilde f_i\|_{L^\infty} \le \delta r^{1-\alpha} \le \delta.$$ 
The Caccioppoli inequality for $\tilde W$ gives
$$ \|\tilde W\|_{H^1(\mathcal S \cap B_{1/2})} \le C. $$ 
Let $W_0$ be the minimizer of the $Q$ functional with the same boundary data as $\tilde W$ on $\mathcal S \cap \p B_{1/2}$, and notice that   $$\| W_0\|_{H^1(\mathcal S \cap B_{1/2})} \le C .$$Then in $\mathcal S \cap B_{1/2}$ we have
\begin{align*}
\langle  \nabla(\tilde W-W_0),\nabla(\tilde W - W_0) \rangle& =\langle \nabla \tilde W, \nabla (\tilde W - W_0) \rangle- \langle \nabla \tilde W_0,\nabla (\tilde W - W_0) \rangle \\
& = \frac 12 \int_{\mathcal S_1} \tilde f_1 (\tilde w_1 - w_{0,1}) dx + \frac 12 \int_{\mathcal S_3} \tilde f_3 (\tilde w_3 - w_{0,3}) dx\\
& \le C \delta \|\tilde W-W_0\|_{L^2}^{\frac 1 2}.
\end{align*}
By Poincar\'e inequality it follows that
$$\|\tilde W - W_0\|_{L^2(\mathcal S \cap B_{1/2})} \le C' \delta^2.$$
The $C^{1,\alpha_0}$ regularity of $W_0$ (see Proposition \ref{C1a}) gives
$$\left(\fint_{\mathcal S \cap B_\rho} |W_0-L_0|^2 dx \right)^{1/2}\le C \rho ^{1+\alpha_0},$$
for some {\it linear} function $L_0$ at $0$. 
The last two inequalities imply the inductive result for $W$ with $r=\rho^{k+1}$ by first choosing $\rho$ sufficiently small depending on $\alpha< \alpha_0$, and then $\delta$ small depending on $\rho$.
\end{proof}

\end{document}